\documentclass[a4paper]{article}

\usepackage[english]{babel}
\usepackage[utf8x]{inputenc}
\usepackage[T1]{fontenc}

\usepackage[a4paper,top=3cm,bottom=2cm,left=3cm,right=3cm,marginparwidth=1.75cm]{geometry}

\usepackage{amsmath}
\usepackage{amssymb}
\usepackage{graphicx}
\usepackage[colorinlistoftodos]{todonotes}
\usepackage[colorlinks=true, allcolors=black]{hyperref}
\usepackage{amsthm}
\usepackage{amsfonts}
\usepackage{MnSymbol}
\usepackage{enumitem}
\usepackage{mathtools}
\usepackage{textgreek}
\usepackage{makecell}
\usepackage{color}
\usepackage{tabularx,booktabs}

\theoremstyle{definition}

\newtheorem*{implementation}{Implementation}
\newtheorem*{notation}{Notation}
\newtheorem{theorem}{Theorem}[section]
\newtheorem{lemma}{Lemma}[section]  
\newtheorem{proposition}{Proposition}[section]

\newtheorem{example}{Example}[section]
\newtheorem{corollary}{Corollary}[section]
\theoremstyle{remark}
\newtheorem*{remark}{Remark}

\newcommand{\q}{\text{ }}

\title{Computer-Assisted Proofs of Congruences for Multipartitions and Divisor Function Convolutions, \\ based on Methods of Differential Algebra}
\author{Alexandru Pascadi}

\begin{document}
\maketitle

\begin{abstract}
This paper provides algebraic proofs for several types of congruences involving the multipartition function and self-convolutions of the divisor function. Our computations use methods of Differential Algebra in $\mathbb{Z}/q\mathbb{Z}$, implemented in a couple of MAPLE programs available as ancillary files on arXiv \cite{Maple1, Maple2}.

The first results of the paper are Ramanujan-type congruences of the form
$p^{*k}(qn+r) \equiv_q 0$ and $\sigma^{*k}(qn+r) \equiv_q 0$, where $p(n)$ and $\sigma(n)$ are the partition and divisor functions, $q > 3$ is prime, and $^{*k}$ denotes $k^{th}$-order self-convolution. We prove all the valid congruences of this form for $q \in \{5, 7, 11\}$, including the three Ramanujan congruences \cite{Ramanujan}, and a nontrivial one for $q = 17$. All such multipartition congruences have already been settled in principle up to a numerical verification due to D. Eichhorn and K. Ono \cite{Ono} via modular forms, but our proofs are purely algebraic. On the other hand, the majority of the divisor function congruences are new results.

We then proceed to search for more general congruences modulo small primes, concerning linear combinations of $\sigma^{*k}(pn+r)$ for different values of $k$, as well as weighted convolutions of $p(n)$ and $\sigma(n)$ with polynomial weights. The paper ends with a few corollaries and extensions for the divisor function congruences, including proofs for three conjectures of N. C. Bonciocat.
\end{abstract}

\section*{Acknowledgements}
The author expresses his gratitude to Professor Terence Tao for his kind support in the preparation of this paper, and to the UCLA Mathematics Department for providing the computing resources necessary to run MAPLE programs with large data.
\section{Introduction}

\begin{notation}
Given a positive integer $n$, we denote by $p(n)$ the number of unordered ways, called partitions, to write $n$ as a sum of positive integers. By convention, $p(0) := 1$ and $p(n) := 0$ for $n < 0$. Also, we denote by $\sigma(n)$ the sum of all the positive divisors of $n$; by convention, $\sigma(n) := 0$ for $n \leq 0$. The connection between these two functions, both of great importance in Number Theory and Combinatorics, arises from a relation between their generating functions defined below:
\begin{align*}
S(X) &:= \sum_{n = 1}^\infty \sigma(n)X^n = X + 3X^2 + 4X^3 + 7X^4 + 6X^5 + \ldots\\
P(X) &:= \sum_{n = 0}^\infty p(n)X^n = 1 + X + 2X^2 + 3X^3 + 5X^4 + \ldots = \prod_{k \geq 1} \left(1 - X^k\right)^{-1}
\end{align*}

The product formula for $P(X)$ is due to Euler and follows from a simple combinatorial argument \cite{Abramo}. To phrase the connection between $S(X)$ and $P(X)$ in a simple manner, let $\partial$ be the operator on formal power series defined by $(\partial F)(X) := X F'(X)$. One can easily check that $\partial$ is a derivation in the meaning from Differential Algebra, i.e. it has the same algebraic properties as differentiation w.r.t. addition and multiplication. Then, we have the following essential equality:

\begin{equation}
S = \frac{\partial P}{P}
\end{equation}
\end{notation}

The proof of this identity is quite straight-forward, using the alleged properties of $\partial$:
\begin{equation}
\begin{split}
\frac{\partial P}{P}(X)
&= \frac{\partial \prod_{k \geq 1} \left(1-X^k\right)^{-1}}{\prod_{k \geq 1} \left(1-X^k\right)^{-1}} \\
&= \sum_{k \geq 1} \frac{\partial \left(1 + X^k + X^{2k} + \ldots\right)}{\left(1-X^k\right)^{-1}} \\
&= \sum_{k \geq 1} \left( \sum_{d \geq 1} dk X^{dk} \right) \left(1-X^k\right) \\
&= \sum_{k \geq 1} \sum_{d \geq 1} d X^{dk} \q= \sum_{n \geq 1} \sigma(n)X^n \q= S(X)
\end{split}
\end{equation}

\begin{remark}
The short proof summarized in the relation above uses a technique equivalent to logarithmic differentiation, except that the standard derivative is replaced with the derivation $\partial$. Throughout this paper, we will prefer $\partial$ over the usual differentiation operator, because it preserves the positions of the coefficients of powers of $X$, in the sense that:

\begin{equation}
\partial\left( \sum_{n = 0}^\infty a_n X^n \right) = \sum_{n = 0}^\infty n a_n X^n
\end{equation}
\end{remark}

Two other essential relations that we will use are Euler's Pentagonal identity \cite{Pentagonal} and an identity of Jacobi \cite{Jacobi}:

\begin{align}
E(X) &:= \sum_{n \in \mathbb{Z}} (-1)^n X^{\frac{1}{2}n(3n+1)} = P(X)^{-1} \\
J(X) &:= \sum_{n \geq 0} (-1)^n (2n + 1) X^{\frac{1}{2} n(n+1)} = E(X)^3 = P(X)^{-3}
\end{align}

There exists a third, more complicated relation of this type due to Winquist \cite{Winquist, JacWin}, concerning $P(X)^{-10}$, but it is not needed for the results of this paper. Such identities of power series with integer coefficients can be reduced modulo a prime $q$, leading to elegant proofs of otherwise difficult congruences. Based on $(4)$ and $(5)$, M. D. Hirschhorn presents in \cite{ShortProof} a simple uniform method to prove all three of Ramanujan's famous partition congruences:
\begin{align}
p(5n + 4) &\equiv_5 0 \\
p(7n + 5) &\equiv_7 0 \\
p(11n + 6) &\equiv_{11} 0
\end{align}

By '$\equiv_q$', we mean congruence modulo $q$; such congruences are understood to hold for all $n \in \mathbb{Z}$.

The method employed by Hirschhorn is based on writing $E = E_0 + \ldots + E_{q-1}$ and $J = J_0 + \ldots + J_{q-1}$, where the components $E_r$ and $J_r$ only contain powers of $X$ whose exponents are congruent to $r$ modulo $q$. As we will prove later using $(4)$ and $(5)$, about half of these component power series are zero modulo $q$, and some clever algebraic manipulations can be used to further evaluate the analogous components of the partition series $P$. For example, one can discover by such arguments that $P_4 \equiv_5 0$, which is equivalent to the first Ramanujan congruence from $(6)$. This idea turns out to be algorithmizable \cite{Algorithm1} and applicable for many more Ramanujan-type congruences.

We will use a different algorithmizable method, based on manipulating differential equations modulo $q$, to obtain a wide class of congruences for multipartitions starting from the same identities in $(4)$ and $(5)$. We note that except for our algebraic methods and the methods employed by Hirschhorn \cite{ShortProof}, most proofs of partition congruences use notions of modular forms \cite{Ono, Ono2, Ono3, Newman, Lazarev}.

Relation $(1)$ will then allow us to deduce similar congruences for self-convolutions of the divisor function, which represents the main advantage of our method over the other available techniques. A secondary advantage is that our differential equations can be used to express and prove more general classes of congruences, involving linear combinations of Ramanujan-type terms as well as weighted convolutions of the partition and divisor functions. However, let us begin with the basics. The simplest class of results proven in this paper are Ramanujan-type congruences of the form:

\begin{equation}
p^{*k}(qn+r) \equiv_q 0 \qquad \text{and} \qquad \sigma^{*k}(qn+r) \equiv_q 0,
\end{equation}

where $q > 3$ is prime and $f^{*k}$ denotes the $k^{th}$ (discrete) self-convolution of a function $f : \mathbb{Z}_{\geq 0} \to \mathbb{C}$, viewed as an infinite sequence of complex numbers.

\begin{remark}
The function $p^{*k}$, also denoted $p_k$, is often called the \emph{multipartition function} or \emph{colored partition function}, and gives the number of partitions of a positive integer into parts colored in $k$ colors. Luckily, self-convolutions of functions like $p$ and $\sigma$ are easy to model with our notations:

\begin{align}
P^k(X) &= \sum_{n = 0}^\infty p^{*k}(n) X^n \\
S^k(X) &= \sum_{n = 0}^\infty \sigma^{*k}(n) X^n
\end{align}
\end{remark}

That being said, the core idea of our method is to rewrite the congruences from $(9)$ as:

\begin{equation}
f(\partial) P^k \equiv_q 0, \qquad\qquad f(\partial) S^k \equiv_q 0,
\end{equation}

where $f$ is a polynomial and $f(\partial)$ is the resulting operator (the congruence of power series is understood coefficient-wise). For instance, Ramanujan's first congruence in $(6)$ can be phrased as:

\begin{equation}
n(n-1)(n-2)(n-3)p(n) \equiv_5 0 \qquad \iff \qquad \partial(\partial - 1)(\partial - 2)(\partial - 3) P \equiv_5 0,
\end{equation}

since the above polynomial cancels out for $n \not\equiv_5 4$. This enables us to translate all the desired congruences modulo $q$ into differential equations with coefficients from $\mathbb{Z}/q\mathbb{Z}$. The problem of proving Ramanujan-type congruences is then reduced to finiding whether a given differential equation like $(13)$ belongs to the (differential) ideal generated by several differential equations derived from $(4)$ and $(5)$, and the derivation $\partial$. While differential Gr\"obner bases would do the trick if we worked in the field of real or complex numbers, ordinary Gr\"obner bases turn out to be preferable when we work modulo $q$ (see Section $3$ for details).

Such multipartition congruences are also studied in a recent article \cite{Lazarev}, which includes the case when the modulus is a prime power. Long before that, D. Eichhorn and K. Ono \cite{Ono} found a way to prove in principle any multipartition congruence of the form
$p^{*k}(tn + r) \equiv_{m} 0$,
where $t$ is an integer and $m$ is a prime power. More precisely, they found a large enough constant such that the congruence is true iff it verifies numerically up to that constant. S. Radu later discovered an algorithmic approach \cite{Algorithm2} that improved the bounds found by Eichhorn and Ono.

\begin{remark}
As opposed to the methods listed above, our differential algebraic approach is not suitable (at least not in an obvious way) for proving congruences modulo prime powers. One reason for this incompatibilty is that many tools of differential algebra work only when the coefficients used lie in a field. Another reason is that it is not always possible to find polynomials modulo prime powers that have a specific set of roots, such as those from relation $(13)$. Nevertheless, the reader is invited to try to work around these problems and generalize the methods presented here.
\end{remark}

There is less literature on the topic of divisor function congruences, where this paper brings most of its original contributions. A few papers \cite{Gandhi4, Ramanathan} study congruences of the form $\sigma(qn - 1) \equiv_q 0$ where $q$ is a divisor of $24$, and others \cite{Gupta} study congruences modulo primes $q > 3$. However, not much is known when we consider convolutions of the divisor function. A primary contribution in this direction was due to N. C. Bonciocat \cite{Bonciocat}, who discovered that:

\begin{equation}
\sigma^{*2}(5n + 1) \equiv_5 0 \qquad \text{and} \qquad \sigma^{*2}(7n + 6) \equiv_7 0
\end{equation}

L. H. Gallardo later found simpler proofs for these congruences \cite{RefBonc}, based on certain identities for convolutions of divisor functions; see \cite{Hahn} for an inventory of such identities. We will extend the results above and prove that the congruence $\sigma^{*k}(qn + r) \equiv_q 0$ holds for $8$ other triplets $(q, k, r)$. In fact, a numerical verification shows that our methods yield all the valid congruences of this form for $q \in \{5, 7\}$, and that for $q \in \{11, 13\}$ there are no such congruences.

The next step is to observe that a much richer class of congruences can be expressed as differential equations modulo $q$. Perhaps the most intuitive generalization of Ramanujan-type congruences arises by taking linear combinations of the terms in $(9)$ (and $(12)$) for various values of $k$. This idea will lead to results such as:

\begin{equation}
\left(\sigma^{*5} + \sigma^{*4}\right)(11n+2) \equiv_{11} 0
\end{equation}

While there are no exceptional congruences of this type for multipartitions, there are many of them for divisor functions - so many, in fact, that the only reasonable way to enumerate them is to provide a basis for their vector space. The following table gathers all the results discussed so far:

\begin{center}
\small
\begin{tabular}{c | c | c | c}
Congruences & $q = 5$ & $q = 7$ & $q = 11$ \\ \hline
\makecell{ \\
$p^{*k}(qn + r) \equiv_q 0$ \\
$\color{gray} ( \text{showing } k \not\equiv_q 0, -1, -3)$ \\
} & \makecell{
$(k, r) = (1, 4)$
} & \makecell{
$(k, r) = (1, 5)$
} & \makecell{
$(k, r) = (1, 6), (3, 7),$ \\
$(5, 8), (7, 9)$ \\
} \\ \hline
\makecell{ \\
$\sigma^{*k}(qn + r) \equiv_q 0$ \\
$\color{gray} (\text{showing } k \not\equiv_q 0)$ \\
} & \makecell{
$(k, r) = (2, 1),$ $\boldsymbol{(3, 1)},$ \\
$\boldsymbol{(3, 2)}, \boldsymbol{(4, 1)}, \boldsymbol{(4, 2)}, \boldsymbol{(4, 3)}$ \\
} & \makecell{
$(k, r) = (2, 6),$ \\
$\boldsymbol{(4, 3)}, \boldsymbol{(4, 6)}, \boldsymbol{(5, 3)}$
} & - \\ \hline

\makecell[t]{ \\
$\sum_{k = 1}^{q-1} c_k \sigma^{*k}(qn + r) \equiv_q 0$ \\ \\
{\color{gray} (showing only a basis} \\
{\color{gray} of their vector space} \\
{\color{gray} for each value of $r$)}
} & \makecell[t]{ \\
$r = 0, (c_1, \ldots, c_4) =$ \\
$\boldsymbol{(1, 0, 0, 2)},$ \\
$\boldsymbol{(0, 1, 0, 1)},$ \\
$\boldsymbol{(0, 0, 1, 3)}$ \\ \\
$r = 1, (c_1, \ldots, c_4) =$ \\
$(0, 1, 0, 0),$ \\
$(0, 0, 1, 0),$ \\
$(0, 0, 0, 1)$ \\ \\
$r = 2, (c_1, \ldots, c_4) =$ \\
$\boldsymbol{(1, 2, 0, 0)},$ \\
$(0, 0, 1, 0),$ \\
$(0, 0, 0, 1)$ \\ \\
$r = 3, (c_1, \ldots, c_4) =$ \\
$\boldsymbol{(1, 0, 1, 0)},$ \\
$\boldsymbol{(0, 1, 4, 0)},$ \\
$(0, 0, 0, 1)$ \\ \\
$r = 4, (c_1, \ldots, c_4) =$ \\
$\boldsymbol{(1, 0, 0, 3)},$ \\
$\boldsymbol{(0, 1, 0, 3)},$ \\
$\boldsymbol{(0, 0, 1, 1)}$ \\
} & \makecell[t]{ \\
$r = 0, (c_1, \ldots, c_6) =$ \\
$\boldsymbol{(1, 0, 0, 5, 4, 0)},$ \\
$\boldsymbol{(0, 1, 0, 1, 2, 0)},$ \\
$\boldsymbol{(0, 0, 1, 4, 3, 0)}$ \\ \\
$r = 1, (c_1, \ldots, c_6) =$ \\
$\boldsymbol{(0, 1, 0, 0, 3, 0)},$ \\
$\boldsymbol{(0, 0, 1, 0, 1, 0)},$ \\
$\boldsymbol{(0, 0, 0, 1, 3, 0)}$ \\ \\
$r = 2, (c_1, \ldots, c_6) =$ \\
$\boldsymbol{(1, 4, 0, 0, 6, 0)},$ \\
$\boldsymbol{(0, 0, 1, 0, 3, 0)},$ \\
$\boldsymbol{(0, 0, 0, 1, 1, 0)}$ \\ \\
$r = 3, (c_1, \ldots, c_6) =$ \\
$\boldsymbol{(1, 0, 3, 0, 0, 0)},$ \\
$\boldsymbol{(0, 1, 1, 0, 0, 0)},$ \\
$(0, 0, 0, 1, 0, 0),$ \\
$(0, 0, 0, 0, 1, 0)$ \\ \\
$r = 4, (c_1, \ldots, c_6) =$ \\
$\boldsymbol{(1, 0, 0, 0, 4, 0)},$ \\
$\boldsymbol{(0, 1, 0, 4, 5, 0)},$ \\
$\boldsymbol{(0, 0, 1, 5, 2, 0)}$ \\ \\
$r = 5, (c_1, \ldots, c_6) =$ \\
$\boldsymbol{(1, 0, 0, 0, 1, 0)},$ \\
$\boldsymbol{(0, 1, 0, 0, 4, 0)},$ \\
$\boldsymbol{(0, 0, 1, 0, 3, 0)},$ \\
$\boldsymbol{(0, 0, 0, 1, 2, 0)}$ \\ \\
$r = 6, (c_1, \ldots, c_6) =$ \\
$\boldsymbol{(1, 0, 0, 0, 2, 0)},$ \\
$(0, 1, 0, 0, 0, 0),$ \\
$\boldsymbol{(0, 0, 1, 0, 6, 0)},$ \\
$(0, 0, 0, 1, 0, 0)$ \\
} & \makecell[t]{ \\
$r = 0, (c_1, \ldots, c_{7}) =$ \\
$\boldsymbol{(1, 0, 10, 7, 3, 3, 7)},$ \\
$\boldsymbol{(0, 1, 6, 1, 6, 9, 10)}$ \\ \\
$r = 1, (c_1, \ldots, c_{7}) =$ \\
$\boldsymbol{(0, 1, 0, 0, 5, 0, 3)},$ \\
$\boldsymbol{(0, 0, 1, 0, 1, 0, 6)},$ \\
$\boldsymbol{(0, 0, 0, 1, 9, 0, 10)}$ \\ \\
$r = 2, (c_1, \ldots, c_{7}) =$ \\
$\boldsymbol{(1, 8, 0, 0, 0, 2, 10)},$ \\
$\boldsymbol{(0, 0, 1, 0, 0, 5, 3)},$ \\
$\boldsymbol{(0, 0, 0, 1, 0, 7, 2)},$ \\
$\boldsymbol{(0, 0, 0, 0, 1, 4, 9)}$ \\ \\
$r = 3, (c_1, \ldots, c_{7}) =$ \\
$\boldsymbol{(1, 0, 7, 0, 6, 6, 4)},$ \\
$\boldsymbol{(0, 1, 5, 0, 9, 5, 7)},$ \\
$\boldsymbol{(0, 0, 0, 1, 4, 1, 8)}$ \\ \\
$r = 4, (c_1, \ldots, c_{7}) =$ \\
$\boldsymbol{(1, 0, 0, 4, 2, 9, 4)},$ \\
$\boldsymbol{(0, 1, 0, 5, 10, 5, 1)},$ \\
$\boldsymbol{(0, 0, 1, 2, 10, 1, 9)}$ \\ \\
$r = 5, (c_1, \ldots, c_{7}) =$ \\
$\boldsymbol{(0, 0, 1, 4, 1, 5, 9)},$ \\
$\boldsymbol{(0, 1, 0, 10, 7, 10, 7)},$ \\
$\boldsymbol{(0, 0, 1, 4, 1, 5, 9)}$ \\ \\
$r = 6, (c_1, \ldots, c_{7}) =$ \\
$\boldsymbol{(1, 0, 0, 0, 4, 5, 5)},$ \\
$\boldsymbol{(0, 1, 0, 0, 9, 4, 4)},$ \\
$\boldsymbol{(0, 0, 1, 0, 4, 7, 7)},$ \\
$\boldsymbol{(0, 0, 0, 1, 0, 7, 7)}$ \\ \\
$r = 7, (c_1, \ldots, c_{7}) =$ \\
$\boldsymbol{(1, 0, 0, 0, 3, 6, 5)},$ \\
$\boldsymbol{(0, 1, 0, 0, 9, 10, 1)},$ \\
$\boldsymbol{(0, 0, 1, 0, 9, 5, 6)},$ \\
$\boldsymbol{(0, 0, 0, 1, 10, 1, 10)}$ \\ \\
$r = 8, (c_1, \ldots, c_{7}) =$ \\
$\boldsymbol{(1, 0, 0, 0, 9, 7, 5)},$ \\
$\boldsymbol{(0, 1, 0, 0, 8, 1, 7)},$ \\
$\boldsymbol{(0, 0, 1, 0, 8, 0, 0)},$ \\
$\boldsymbol{(0, 0, 0, 1, 0, 8, 1)}$ \\ \\
$r = 9, (c_1, \ldots, c_{7}) =$ \\
$\boldsymbol{(1, 0, 0, 6, 10, 9, 7)},$ \\
$\boldsymbol{(0, 1, 0, 5, 8, 4, 8)},$ \\
$\boldsymbol{(0, 0, 1, 4, 10, 2, 4)}$ \\ \\
$r = 10, (c_1, \ldots, c_{7}) =$ \\
$\boldsymbol{(1, 0, 0, 0, 9, 4, 1)},$ \\
$\boldsymbol{(0, 1, 0, 0, 9, 5, 4)},$ \\
$\boldsymbol{(0, 0, 1, 0, 0, 6, 7)},$ \\
$\boldsymbol{(0, 0, 0, 1, 3, 6, 7)}$
}
\end{tabular}
\end{center}

\begin{remark} Firstly, in the former table, the results that are new to the best of the author's knowledge are shown in bold (only once if they appear in multiple equivalent forms). Secondly, in the first two rows of the table, adding a multiple of the prime $q$ to the convolution exponent $k$ preserves the congruences (see Lemma $3.2$). Thirdly, a large number of trivial congruences are not shown (e.g., when $k \equiv_q 0$), but can be found easily by following our methods. \textbf{Also, in the presentation of the vector space bases for $\boldsymbol{q = 11}$, the coefficients $\boldsymbol{c_8, c_9, c_{10}}$ are omitted for notation brevity, since they are always 0} (i.e., their corresponding terms are irrelevant in all congruences).
\end{remark}

\begin{example}
The quadruple $(c_1, \ldots, c_4) = (0, 1, 4, 0)$ for $r = 3$ in the third line, first column of tuples from the former table should be read as the following congruence:

\[
\left(\sigma^{*2} + 4\sigma^{*3}\right)(5n + 3) \equiv_5 0
\]

In other words, for any integer $n \equiv_5 3$, one has:

\[
\sigma^{*2}(n) \equiv_5 \sigma^{*3}(n).
\]
\end{example}

To these results we should add the following congruence modulo $17$:

\begin{equation}
p^{*13}(17n + 14) \equiv_{17} 0
\end{equation}

Normally, this relation would fit in the first line of the table from the previous page, on a column dedicated to the prime $q = 17$. However, the rest of that column would be empty, so we chose not to include $(16)$ in the table for reasons of aesthetics. For a thorough study of such congruences modulo powers of $17$, see K. Hughes \cite{17powers}; for similar studies modulo powers of other primes, see \cite{Chen, Dazhao, Wang}. It is possible in principle to prove other congruences for $q = 17$ (or for bigger primes) with our methods, but our algorithms take too much time and memory. For $q = 13$, a numerical generation shows that there are no exceptional congruences to be found.

Another class of congruences generalizing those of Ramanujan involves weighted convolutions:

\begin{notation}
Given a function $f : \mathbb{Z}_{\geq 0} \to \mathbb{C}$, a positive integer $k$ and a so-called weight function $w : {\mathbb{Z}_{\geq 0}}^k \to \mathbb{C}$, we denote:

\begin{equation}
f^{*k}_w (n) := \sum_{\substack{a_1, \ldots, a_k \in\q \mathbb{Z}_{\geq 0} \\ a_1 + \ldots + a_k = n}} w(a_1, \ldots, a_k) f(a_1)\cdot \ldots \cdot f(a_k)
\end{equation}

Note that for $w \equiv 1$, $f^{*k}_w = f^{*k}$. As a simple example, relation $(13)$ can be rewritten as $p^{*1}_{w} \equiv_5 0$, where $w(x) = x(x-1)(x-2)(x-3)$.
\end{notation}

Using the same methods of differential algebra, it is possible to prove congruences of the form $p^{*k}_w \equiv_q 0$ and $\sigma^{*k}_w \equiv_q 0$, where $w$ is a multivariate polyonomial (of course, since the argument of the functions is omitted, these congruences are understood to hold when evaluated at any integer). Due to the high dimensionality of the weight function's parameters, such congruences are in fact much harder to discover than to prove. In Subsection $5.3$, we analyze a fraction of these congruences for when $k = 2$, to obtain the results listed in the table below:

\begin{center}
\small
\begin{tabular}{c | c | c}
Congruences & $\text{Weights } w(a, b) \text{ that work for } q = 5$ & $\text{Weights } w(a, b) \text{ that work for } q = 7$ \\ \hline
\makecell{
$p^{*2}_w(n) \equiv_q 0$ \\
\color{gray} (showing $w$ consisting of \\
\color{gray} at most $q - 2$ terms) \\
} &
\makecell{ \\
$\boldsymbol{a^2 + ab + 4a},$ \\ \\
$\boldsymbol{a^3 + 3a^2b + 4a},$ \\ \\
$\boldsymbol{a^4 + 4a^3 + 4ab}$ \\ \\
} & 
\makecell{ \\
$\boldsymbol{a^6 + a^5 + 6a^4 + 6a^2 + 2ab},$ \\ \\
$\boldsymbol{a^6 + 3a^4 + 2a^3 + 6ab + a}$ \\ \\
}
\\ \hline

\makecell{
$\sigma^{*2}_w(n) \equiv_q 0$ \\
\color{gray} (showing $w$ consisting of \\
\color{gray} at most $q - 2$ terms) \\
} &
\makecell{ \\
$\boldsymbol{a^2 + 3ab + a},$ \\ \\
$\boldsymbol{a^3b + a^2b^2 + 3ab},$ \\ \\
$\boldsymbol{a^4b + 3a^3b^2 + a^2b},$ \\ \\
$\boldsymbol{a^4b^2 + 2a^3b^3 + 2ab}$ \\ \\
} & 
\makecell{ \\
$\boldsymbol{a^6b^2 + 6a^5b^3 + 4a^4b^4 + 3ab}$ \\ \\
}
\\

\end{tabular}
\end{center}

\begin{example}
The weight polynomial function $a^6 + a^5 + 6a^4 + 6a^2 + 2ab$ (from the first line, second column of weights from the previous table) summarizes the following congruence:

\[
\sum_{a+b=n} a(a^5 + a^4 + 6a^3 + 6a + 2b) \q p(a)p(b) \equiv_7 0
\]
\end{example}

In the end of the paper, we explore some related congruences and conjectures concerning (convolutions of) divisor functions, in particular the odd and even divisor functions. Our results include three conjectures of N. C. Bonciocat \cite{Bonciocat}, as well as other miscellaneous congruences such as:

\begin{equation}
\sum_{\substack{a+b = n \\ a, b \geq 1}} \sigma(ab) \equiv_5
\begin{cases}
0, & n \equiv_5 1, 4\\
1, & n \equiv_5 2, 3\\
3, & n \equiv_5 0
\end{cases}
\end{equation}

(for any positive integer $n$).

\section{Two Illustrative Proofs}

Here we prove Ramanujan's first two congruences to illustrate our methods. Unlike the algorithms presented in the next section, which use both relations $(4)$ and $(5)$, these proofs will only need Jacobi's identity:

\begin{equation}
P(X)^{-3} = \sum_{n \geq 0} (-1)^n (2n + 1) X^{\frac{1}{2} n(n+1)}
\end{equation}

Additionally, we will take some shortcuts here that the algorithmic method cannot generally apply, but the main idea of the proofs is the same.

\begin{theorem}[Ramanujan's First Congruence]
For any (nonnegative) integer $n$, one has:

\begin{equation}
p(5n + 4) \equiv_5 0
\end{equation}
\end{theorem}

\begin{proof}
In this proof, we use the usual differentiation operator instead of the derivation $\partial$ defined in the introduction. While $\partial$ is more helpful in the general case, standard differentiation leads to shorter proofs in the particular case when $q = 5$. That being said, note that in equation $(19)$, the value $\frac{1}{2}n(n+1)$ can only assume the residues $0$, $1$ and $3$ modulo $5$. Moreover, when it assumes the residue $3$, one must have $n \equiv_5 2$, which cancels the factor $2n+1$ modulo $5$. Therefore, the powers of $X$ with exponents that are congruent to $2$, $3$ or $4$ modulo $5$ have coefficients divisible by $5$ in the RHS of $(19)$. In the language of differential equations, this is equivalent to the fact that:

\begin{equation}
\left(P^{-3}\right)'' \equiv_5 0
\end{equation}

Indeed, the double differentiation applied to each $X^k$ will bring down the coefficient $k(k-1)$, which cancels out modulo $5$ when $k \equiv_5 0, 1$. Since the other coefficients are divisible by $5$ anyway, the whole series is congruent to $0$ mod $5$. Now from $(21)$ it follows that:

\begin{equation}
\left(P^2\right)'' = \left(P^5 \cdot P^{-3}\right)'' \equiv_5 P^5\left(P^{-3}\right)'' \equiv_5 0
\end{equation}

The factor $P^5$ jumps out the parentheses above because $\left(P^5\right)' = 5P^4 \equiv_5 0$. Rewrite the LHS as:

\begin{equation}
\left(P^2 \right)'' = (2PP')' = 2\left( (P')^2 + PP'' \right) \equiv_5 0
\end{equation}

Write $T := \frac{P}{P'}$. Since we can divide by $P$, we conclude that:

\begin{equation}
- P' \equiv_5 \frac{P}{P'} P'' = TP^{(2)}
\end{equation}

Moreover, one can see that:

\begin{equation}
T' = \frac{(P')^2 - P P''}{(P')^2} \equiv_5 \frac{2(P')^2}{(P')^2} \equiv_5 2
\end{equation}

We proceed to differentiate relation $(24)$ twice using the fact that $T' \equiv_5 2$:
\begin{align}
-P' &\equiv_5 TP^{(2)} \\
-P^{(2)} &\equiv_5 2P^{(2)} + TP^{(3)} \\
-P^{(3)} &\equiv_5 2P^{(3)} + 2P^{(3)} + TP^{(4)}
\end{align}

Therefore:
\begin{equation}
TP^{(4)} \equiv_5 -5P^{(3)} \equiv_5 0 \q\quad \Rightarrow \quad P^{(4)} \equiv_5 0
\end{equation}

So every monomial of the form $X^{5n + 4}$ has a coefficient divisible by $5$ in $P(X)$, i.e. $p(5n+4) \equiv_5 0$.

\end{proof}

The proof of the second Ramanujan congruence will illustrate why we cannot count on the standard differentiation operator in the general case:

\begin{theorem}[Ramanujan's Second Congruence]
For any (nonnegative) integer $n$, one has:

\begin{equation}
p(7n + 5) \equiv_5 0
\end{equation}
\end{theorem}

\begin{proof}
As in the start of the previous proof, we remark that the value $\frac{1}{2}n(n+1)$ can only assume the residues $0$, $1$, $3$ and $6$ modulo $7$. Furthermore, when it assumes the residue $6$, we must have $n \equiv_7 3$, which cancels the factor $2n+1$ modulo $7$. Hence the powers of $X$ with exponents that are congruent to $2$, $4$, $5$ or $6$ modulo $7$ occur with coefficients divisible by $7$ in the RHS of relation $(19)$. Using the derivation operator $\partial$, we can express this fact as:

\begin{equation}
\partial(\partial - 1)(\partial - 3)\left(P^{-3}\right) \equiv_7 0
\end{equation}

Indeed, the differential operator $\partial(\partial - 1)(\partial - 3)$ will cancel out (mod $7$) all powers of $X$ whose exponents are congruent to $0$, $1$ or $3$ modulo $7$, and the other powers have coefficients divisible by $7$ anyway. This is the same idea that led to equation $(13)$. Now by the same reasoning as in the previous proof, a factor of $P^7$ applied to the LHS above will jump right into the argument $P^{-3}$ of the differential operator, since:

\begin{equation}
\partial P^7(X) = 7X P^6(X) \equiv_7 0
\end{equation}

Therefore:
\begin{equation}
\partial(\partial - 1)(\partial - 3)\left(P^4\right) \equiv_7 0
\end{equation}

In the general case, we (our algorithms) would expand this relation and differentiate it several times to construct a Gr\"obner basis for an ideal of valid differential equations. However, we will take a shortcut here and divide the given equation by $Q := P^4$ (which is clearly invertible):

\begin{equation}
\frac{\partial^3Q}{Q} + 3\frac{\partial^2Q}{Q} + 3\frac{\partial Q}{Q} \equiv_7 0
\end{equation}

Since $\partial$ is a derivation and $Q = P^4$, one can infer that:
\begin{align}
\frac{\partial Q}{Q} &= 4\frac{\partial P}{P} = 4S \\
\frac{\partial^2 Q}{Q} &= \left(\partial + \frac{\partial Q}{Q}\right)\left(\frac{\partial Q}{Q}\right) \equiv_7 2S^2 + 4\partial S \\
\frac{\partial^3 Q}{Q} &= \left(\partial + \frac{\partial Q}{Q}\right)\left(\frac{\partial^2 Q}{Q}\right) \equiv_7 S^3 + 6S\partial S + 4\partial^2 S
\end{align}

For a general way to express $\frac{\partial^n Q}{Q}$ in terms of $S$, see the proof of Theorem $3.2$. Plugging these values into $(34)$ gives us that:

\begin{equation}
S^3 \equiv_7 S^2 + S\partial S + 2S + 2\partial S + 3\partial^2 S,
\end{equation}

which itself represents a congruence involving divisor function convolutions (but we will worry about that later).

Next, by the same reasoning as in the progression of equations $(35)$, $(36)$, $(37)$, one finds that:
\begin{align}
\frac{\partial P}{P} &= S \\
\frac{\partial^2 P}{P} &= \left(\partial + \frac{\partial P}{P}\right)\left(\frac{\partial P}{P}\right) = S^2 + \partial S \\
\frac{\partial^3 P}{P} &= \left(\partial + \frac{\partial P}{P}\right)\left(\frac{\partial^2 P}{P}\right) = S^3 + 3S\partial S + \partial^2 S
\end{align}

Plugging relations $(38)$, $(39)$ and $(40)$ into $(41)$ eventually leads to:

\begin{equation}
\left(\partial^3 - \partial^2 - \partial\right) P = (4\partial + 1)(P\partial S)
\end{equation}

Since $4 \cdot 5 + 1 \equiv_7 0$, all powers of $X$ with exponents congruent to $5$ modulo $7$ are canceled out in the RHS of $(42)$. However, the polynomial $\partial^3 - \partial^2 - \partial$ in the LHS does not cancel out these exponents, since $5^3 - 5^2 - 5 \equiv_7 4$. We conclude that all such exponents have coefficients divisible by $7$ in the power series $P(X)$. Equivalently, $p(7n+5) \equiv_7 0$, for all (nonnegative) integers $n$.

\end{proof}

\begin{remark}
There is no simple way to express equation $(31)$ using the standard differentiation operator, because the residues $0$, $1$ and $3$ are not consecutive. Even if we split $\partial(\partial - 1)(\partial - 3)$ into a sum of polynomials that arise from standard differentiation, e.g. $\partial(\partial - 1)(\partial - 2) - \partial(\partial - 1)$, these two polynomials correspond to different orders of differentiation. Since each differentiation subtracts $1$ from the exponent of each power of $X$, the resulting exponents would be shifted, so we cannot equate $(\partial(\partial - 1)(\partial - 2) - \partial(\partial - 1))(F) \equiv_7 0$ to $F^{(3)} - F^{(2)} \equiv_7 0$, given a power series $F$. Indeed, writing equation $(31)$ properly in terms of the usual differentiation operator would require some multiplications by powers of $X$ to make up for the shifted exponents, which would significantly complicate our equations and make an algorithmic approach much more difficult.
\end{remark}

If we tried to prove other similar congruences by hand, the differential algebraic computations would become too convoluted. Luckily, the methods illustrated in this section (with slight variations) are algorithmizable. 

\section{Method and Structure of the Algorithms}

Suppose we want to find congruences modulo a prime $q > 3$. As specified in the introduction, our methods are based on exploiting relations $(4)$ and $(5)$ modulo $q$.  The following lemma is helpful:

\begin{lemma}
Given a polynomial $f \in (\mathbb{Z}/q\mathbb{Z})[X]$ of degree $2$, the image $f(\mathbb{Z}/q\mathbb{Z})$ contains exactly $\frac{q+1}{2}$ elements.
\end{lemma}

\begin{proof}
Write $f(X) = \hat{a}X^2 + \hat{b}X + \hat{c}$ with $\hat{a} \neq \hat{0}$ (here, hats denote equivalence classes modulo $q$). Assume WLOG that $\hat{c} = \hat{0}$, $\hat{a} = \hat{1}$ (translations and scalings won't affect the cardinality of $f(\mathbb{Z}/q\mathbb{Z})$ since $\mathbb{Z}/q\mathbb{Z}$ is a field), so $f(X) = X^2 + \hat{b}X$. Now if $f(\hat{m}) = f(\hat{n})$ for some $m, n \in \mathbb{Z}$, we get:

\begin{equation}
\hat{b}(\hat{n} - \hat{m}) = \hat{m}^2 - \hat{n}^2 = (\hat{m} + \hat{n})(\hat{m} - \hat{n})
\end{equation}

Thus either $\hat{m} = \hat{n}$, or we can divide by $\hat{m} - \hat{n}$ to get:

\begin{equation}
\hat{m} + \hat{n} = -\hat{b}
\end{equation}

Note that the conditions $\hat{m} = \hat{n}$ and $\hat{m} + \hat{n} = -\hat{b}$ occur simultaneously for exactly one value of $\hat{m}$, i.e. $\hat{m} = \frac{-\hat{b}}{\hat{2}}$. (we can divide by $\hat{2}$ since $2 \not\divides q$). The other elements of $\mathbb{Z}/q\mathbb{Z}$ can be partitioned into pairs $\{\hat{m}, - \hat{b} - \hat{m}\}$, and each pair is mapped to a distinct value by $f$. Thus there are $1 + \frac{q-1}{2} = \frac{q+1}{2}$ elements in the image $f(\mathbb{Z}/q\mathbb{Z})$.
\end{proof}

\begin{corollary}
There are exactly $\frac{q+1}{2}$, respectively $\frac{q-1}{2}$ residues $r \in \{0, \ldots, q-1\}$ such that:

\begin{equation}
\left[X^n\right](E(X)) \not\equiv_q 0, \qquad \text{for some } n \equiv_q r,
\end{equation}

respectively:
\begin{equation}
\left[X^n\right](J(X)) \not\equiv_q 0, \qquad \text{for some } n \equiv_q r
\end{equation}
\end{corollary}

\begin{proof}
The first statement follows by considering $f(X) := \frac{X(\hat{3}X+\hat{1})}{\hat{2}}$, in light of Lemma $3.1$ and relation $(4)$; note that $f$ has degree $2$ because $q \neq 3$. For the second statement, consider the polynomial $f(X) := \frac{X(X+\hat{1})}{\hat{2}}$. By Lemma $3.1$, there are $\frac{q+1}{2}$ elements in $f(\mathbb{Z}/q\mathbb{Z})$, among which the class $\frac{-\hat{1}}{\hat{8}}$ is the image of exactly one element: $\frac{-\hat{1}}{\hat{2}}$. However, for $n \equiv_q \frac{q-1}{2}$, we have $2n + 1 \equiv_q 0$, so $[X^n](J(X)) \equiv_q 0$ by $(5)$. This leaves only $\frac{q+1}{2} - 1 = \frac{q-1}{2}$ residues modulo $q$ that are covered by the exponents in $J(X)$ corresponding to nonzero coefficients, which completes our proof.
\end{proof}

\begin{remark}
Corollary $3.1$ can be understood as a set of congruences of the form $p^{*k}(qn + r) \equiv_q 0$; to be more precise, the corollary reveals $q - \frac{q+1}{2} = \frac{q-1}{2}$ such congruences for $k = -1$ and $q - \frac{q-1}{2} = \frac{q+1}{2}$ such congruences for $k = -3$, as in the following example:
\end{remark}

\begin{example}
For $q = 5$, one manually checks that relation $(45)$ holds only for $r \in \{0, 1, 2\}$, whereas relation $(46)$ holds only for $r \in \{0, 1\}$. Thus, one has:
\begin{align}
[X^n](E(X)) &\equiv_5 0, \qquad \forall n \equiv_5 3, 4 \\
[X^n](J(X)) &\equiv_5 0, \qquad \forall n \equiv_5 2, 3, 4
\end{align}

Since $E = P^{-1}$ and $J = P^{-3}$, these relations can be translated into Ramanujan-type congruences:
\begin{align}
p^{*-1}(5n + r) &\equiv_5 0, \qquad \text{for } r \in \{3, 4\}, \forall n \\
p^{*-3}(5n + r) &\equiv_5 0, \qquad \text{for } r \in \{2, 3, 4\}, \forall n 
\end{align}

\end{example}

In fact, due to the following lemma, the congruences above have immediate generalizations:

\begin{lemma}
Suppose $f : \mathbb{Z}_{\geq 0} \to \mathbb{Z}$ satisfies $f^{*k}(qn + r) \equiv_q 0$ for all $n \geq 0$, where $k\in \mathbb{Z}$ and $r \in \{0, \ldots, q-1\}$ are given. Then one can replace $k$ by any $k' \equiv_q k$ and the congruence will still hold, provided $f^{*k'}$ is defined. That is, adding a multiple of $q$ to $k$ preserves the congruence.
\end{lemma}

\begin{proof}
Consider the formal power series $F(X) := \sum_{n \geq 0} f(n)X^n$. Then one can write $F^{k'} = F^k G^q$, where $G := F^{\frac{1}{q}(k'-k)}$. A classical proof of this lemma would continue to study the coefficients of $G^q$ modulo $q$ (see \emph{Freshman's Dream}). However, to introduce a method that will be essential to our results later, we give a different proof here. The congruence $f^{*k}(qn + r) \equiv_q 0$ can be rephrased as the following differential equation over $\mathbb{Z}/q\mathbb{Z}$:

\begin{equation}
\left(\prod_{\substack{0 \leq s < q \\ s \neq r}} (\partial - \hat{s})\right) \left(F^{k}\right) = \hat{0},
\end{equation}

where it is understood that the coefficients of $F^k$ are taken modulo $q$. Indeed, if $n \not\equiv_q r$, then the coefficient of $X^n$ in the LHS of $(51)$ (when applied to $X$) must be divisible by $q$ because the polynomial $\prod_{\substack{0 \leq s < q \\ s \neq r}} (n - s)$ cancels modulo $q$; recall the meaning of the derivation $\partial$ from the introduction. Thus $(51)$ is equivalent to the fact that for $n \equiv_q r$, the coefficient of $X^n$ in the LHS is divisible by $q$. Since $\prod_{\substack{0 \leq s < q \\ s \neq r}} (n - s) \not\equiv_q 0$ in this case, the latter is in its turn equivalent to $[X^n](F^k(X))$ being divisible by $q$ for all $n \equiv_q r$, which is precisely our initial congruence. Now note that given any formal power series $H$:

\begin{equation}
\partial \left(H G^q \right) = (\partial H) G^q + H\partial\left(G^q\right) = (\partial H)G^q \quad \text{(in $\mathbb{Z}/q\mathbb{Z}$)}
\end{equation}

So the factor $G^q$ jumps over the derivation $\partial$, and thus over any polynomial in this derivation. Since $F^{k'} = F^k G^q$, multiplying $(51)$ by $G^q$ yields that:

\begin{equation}
\left(\prod_{\substack{0 \leq s < q \\ s \neq r}} (\partial - \hat{s})\right) \left(F^{k'}\right) = \hat{0},
\end{equation}

which is equivalent to the alleged congruence $f^{*k'}(qn + r) \equiv_q 0$.
\end{proof}

Returning to Example $3.1$ and the prior remark, Lemma $3.2$ shows that congruences $(49)$ and $(50)$, together with their analogues for any prime $q > 3$, hold if the convolution exponents $-1$ and $-3$ are replaced by any $k \equiv_q -1$, respectively $k \equiv_q -3$. The resulting relations, along with the trivial congruences obtained for $k \equiv_q 0$ (which follow because $E^0 = 1$), are sometimes called \emph{non-exceptional} \cite{Lazarev}. In contrast, all other congruences of the form $p^{*k}(qn + r) \equiv_q 0$, such as the ones shown in the table from the introduction, are called \emph{exceptional} and given more importance.

In order to find and prove the exceptional Ramanujan-type congruences, as well as more general types of results, we turn to the following theorem:

\begin{theorem}
There exist two polynomial differential equations $B_1(E) = \hat{0}$ and $B_2(E) = \hat{0}$ over the field $\mathbb{Z}/q\mathbb{Z}$, derived from relations $(4)$ and respectively $(5)$, such that:
\begin{itemize}
\item Each monomial in $B_1$ has degree $1$ and at most $\frac{q+1}{2}$ total applications of the derivation $\partial$.
\item Each monomial in $B_2$ has degree $3$ and at most $\frac{q-1}{2}$ total applications of the derivation $\partial$.
\end{itemize}
In other words, both of these equations have relatively simple forms, ideal for performing intricate machine computations. We will refer to them as the \emph{base relations} of our method, comprising in some sense the input of our algorithms. When we say \emph{polynomial differential equations}, we mean that $B_1$ and $B_2$ are multivariate polynomials in $E$, $\partial E$, $\partial^2 E$, etc., and the degrees of the monomials in $B_1$ and $B_2$ are understood in this sense (e.g., $E \partial^2 E$ has degree 2).
\end{theorem}

\begin{proof}
A way of rephrasing Corollary $3.1$, following the same idea as in Lemma $3.2$, is to say that there exist residues $r_1, \ldots, r_{\frac{1}{2}(q+1)}$ and $s_1, \ldots, s_{\frac{1}{2}(q-1)}$ such that in $\mathbb{Z}/q\mathbb{Z}$:

\begin{equation}
\left(\prod_{j = 1}^{\frac{1}{2}(q+1)} (\partial - \hat{r_j})\right) E = \hat{0}
\end{equation}

\begin{equation}
\left(\prod_{j = 1}^{\frac{1}{2}(q-1)} (\partial - \hat{s_j})\right) J = \hat{0}
\end{equation}

The two relations above describe the desired polynomial differential equations, after substituting $J = E^{3}$. The restrictions from the hypothesis of the theorem are easy to verify.
\end{proof}

\begin{remark}
If we allowed the prime $q$ to be equal to $3$, then equation $(55)$ would provide no valuable information. Indeed, since $J = E^3$, we would have $\partial J = \hat{0}$, making equation $(55)$ trivial.
\end{remark}

\begin{example}
In continuation of Example $3.1$, relations $(54)$ and $(55)$ are given when $q = 5$ by:

\begin{equation}
\begin{cases}
\partial(\partial - \hat{1})(\partial - \hat{2}) E = \hat{0} \\
\partial(\partial - \hat{1}) J = \hat{0},
\end{cases}
\end{equation}

which after simplification become:

\begin{equation}
\begin{cases}
\partial^3 E + 2\partial^2 E + 2\partial E = \hat{0} \\
(\partial E)^2 + \hat{3}E \partial^2 E + \hat{2}E \partial E = \hat{0}
\end{cases}
\end{equation}

We note that the second equation has been divided by $E$ (which is okay since $E$ is not the zero power series). In this particular case, it turns out that the first equation in $(57)$ can be deduced from the second one, but in general the two equations give different information.
\end{example}

\begin{implementation}
The MAPLE code that produces the differential equations from Theorem $3.1$, given the information about residues from Corollary $3.1$ in the form of a boolean array called \emph{reached}, is found within the \emph{buildEquation} procedure from the file at \cite{Maple1}. The resulting so-called \emph{base relations} are denoted in the program by the variables \emph{baseE} and \emph{baseJ}. The value of $q$ is declared at the very beginning of the program.
\end{implementation}

\begin{remark}
In the Differential Algebra package of MAPLE, the application of a derivation $t$ to a variable $F$ is denoted by $F[t]$. Multiple consecutive applications of the same derivation to $F$ are denoted by $F[t, t, \ldots, t]$. For more information on specific syntax, see:

\begin{center}
\url{https://www.maplesoft.com/support/help/maple/view.aspx?path=DifferentialAlgebra}
\end{center}
\end{remark}

While Theorem $3.1$ provides two polynomial differential equations in the "variable" $E$, which will help prove congruences for multipartitions, the following theorem provides differential equations in the "variable" $S$, which will help prove congruences for convolutions of the divisor function:

\begin{theorem}
There exist two polynomial differential equations $B_3(S) = \hat{0}$ and $B_4(S) = \hat{0}$ over the field $\mathbb{Z}/q\mathbb{Z}$, derived from relations $(4)$ and respectively $(5)$, such that:
\begin{itemize}
\item Each monomial in $B_3$ has degree at most $\frac{q+1}{2}$ and at most $\frac{q-1}{2}$ total applications of the derivation $\partial$.
\item Each monomial in $B_4$ has degree at most $\frac{q-1}{2}$ and at most $\frac{q-3}{2}$ total applications of the derivation $\partial$.
\end{itemize}

As anticipated above, these equations will play the role of \emph{base relations} in our computations when we try prove congruences for divisor function convolutions. Note that these base relations can get significantly more complicated than the ones from Theorem $3.1$ (due to weaker restrictions), which is why obtaining congruences for multipartitions will generally be more efficient.
\end{theorem}

\begin{proof}
Suppose $Q = P^k$ for some integer $k$ (e.g., $Q = P$ if $k = 1$, $Q = E$ if $k = -1$, and $Q = J$ if $k = -3$). By the properties of derivations, equation $(1)$ implies that:

\begin{equation}
\frac{\partial Q}{Q} = \frac{kP^{k-1}\partial P}{P^k} = kS
\end{equation}

Our goal is to express any $\frac{\partial^n Q}{Q}$, for any power $n \geq 1$, as a differential polynomial of $S$. The case $n = 1$ is given by the equation above, and the inductive step follows from the observation that:

\begin{equation}
\begin{split}
\frac{\partial^{n+1} Q}{Q}
&= \frac{Q\partial^{n+1} Q - \partial^n Q \partial Q}{Q^2} + \frac{\partial^{n}Q \partial Q}{Q} \\
&= (\partial + kS)\left(\frac{\partial^n Q}{Q}\right)
\end{split}
\end{equation}

Thus in general, for any $n \geq 1$:

\begin{equation}
\frac{\partial^{n} Q}{Q} = (\partial + kS)^{n-1} (kS)
\end{equation}

As a side note, this is the same reasoning used in the proof of Theorem $2.2$ (see $(35)-(37)$, $(39)-(41)$). Consider now the equations $(54)$ and $(55)$, rewritten in the forms:
\begin{align}
\sum_{j = 1}^{\frac{q+1}{2}} \hat{c_j} \partial^j E &= \hat{0}, \\
\sum_{j = 1}^{\frac{q-1}{2}} \hat{d_j} \partial^j J &= \hat{0},
\end{align}

for suitable coefficients $c_j$ and $d_j$. We note that $j$ starts at $1$ because the polynomials in $\partial$ from $(54)$ and $(55)$ have no free coefficient; this occurs because $0$ is always among both the $r_j$'s and the $s_j$'s, since $\frac{0(3\cdot 0 + 1)}{2} = \frac{0(0 + 1)}{2} = 0$ (see the proof of Corollary $3.1$).

Now divide the equations above by $E = P^{-1}$, respectively $J = P^{-3}$ to get that:
\begin{align}
\sum_{j = 1}^{\frac{q+1}{2}} \hat{c_j} (\partial - S)^{j-1} (-S) &= \hat{0}, \\
\sum_{j = 1}^{\frac{q-1}{2}} \hat{d_j} (\partial - 3S)^{j-1} (-\hat{3}S) &= \hat{0},
\end{align}

which can be taken as the desired polynomial differential equations in $S$. The restrictions from the hypothesis of the theorem are easy to verify.
\end{proof}

\begin{remark}
If we allowed $q$ to be equal to $3$, relation $(64)$ would just tell us that $0 = 0$. Therefore, our algorithms do not work for $q = 3$ (they either crash or lead to trivial results).
\end{remark}

\begin{example}
In continuation of Examples $3.1$ and $3.2$ (where $q = 5$), the relations in $(57)$ are equivalent (by the method from the proof of Theorem $3.2$) to:

\begin{equation}
\begin{cases}
\hat{4}S^3 + \hat{2}S^2 + \hat{3}S\partial S + \hat{3}S + \hat{3}\partial S + \hat{4}\partial^2 S = \hat{0} \\
\hat{4}S^2 + \hat{3}S + \hat{2}\partial S = \hat{0}
\end{cases}
\end{equation}
\end{example}

\begin{implementation}
The MAPLE code that determines the base relations from Theorem $3.2$, given the residues from Corollary $3.1$ in the form of a boolean array called \emph{reached}, is found within the \emph{buildEquation} procedure from the file at \cite{Maple2}. Once again, the resulting relations are denoted in the program by the variables \emph{baseE} and \emph{baseJ}, and the value of $q$ is declared at the very top.
\end{implementation}

Now that we have acquired a sufficient set of base relations, it is time to see how these can be combined to produce a predetermined so-called \emph{target relation}. The specific form of the target relation will depend on the type of congruence that we want to prove (see the next section for details); all we know for now is that it is a polynomial differential equation over $\mathbb{Z}/q\mathbb{Z}$, either $T(P) = \hat{0}$ if we are searching for a congruence of multipartitions, or $T(S) = \hat{0}$ if we are searching for a congruence of divisor function convolutions. The figure below illustrates the grand scheme of our algorithms:

\begin{center}
\includegraphics[scale = 0.9]{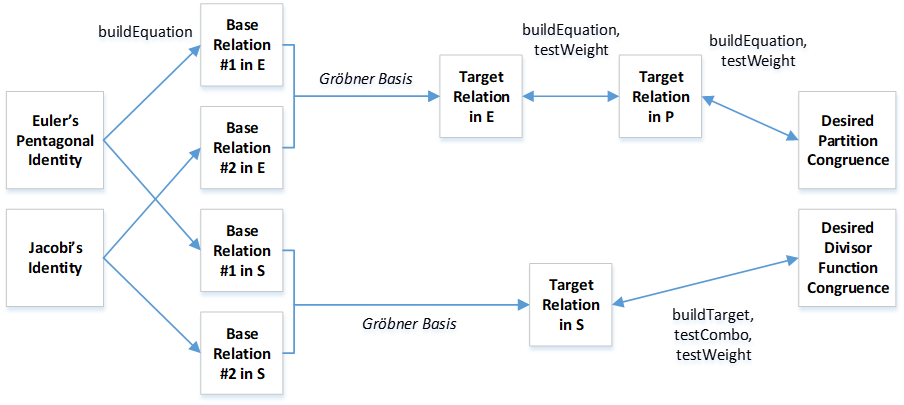}
\end{center}

\begin{implementation}
The text next to each arrow above indicates which MAPLE procedure from our code is used to perform that transition. More details on the procedures \emph{buildTarget}, \emph{testCombo} and \emph{testWeight} are given in the next section. The top chain of transitions from the figure (dedicated to partition congruences) corresponds to the program available at \cite{Maple1}, while the bottom chain (dedicated to divisor function congruences) corresponds to the program available at \cite{Maple2}.
\end{implementation}

\begin{remark}
The conversion of a polynomial differential equation in terms of $P$, say $T(P) = \hat{0}$, to one in terms of $E$, say $U(E) = \hat{0}$, can always be realized through the method further described. Suppose that the largest monomial in $T(P)$ (which is a product of $P$, $\partial P$, etc.) has degree $k$. Then one can multiply the relation $T(P) = \hat{0}$ by $E^{qk}$, which is the same as $P^{-qk}$. This way, each component $\left(\partial^j P\right)^n$ of each monomial in $T$ can absorb a factor of $P^{-qn}$, eventually leaving a leftover power of $E$ as a factor to each monomial. By the reasoning from the proof of Lemma $3.2$, the resulting component will become $\left(P^{-q} \partial^j P\right)^n = \left(\partial^j P^{1-q}\right)^n = \left(\partial^j E^{q-1}\right)^n$. The expression obtained in this way is clearly a differential polynomial in $E$, giving us the desired relation $U(E) = \hat{0}$. In particular, this reasoning shows that if all monomials in $T(P)$ have the same degree, one can replace $P$ by $E^{q-1}$ directly to obtain that $U(E) := T(E^{q-1}) = \hat{0}$. For computational reasons, the differential polynomial $U(E)$ should be simplified by the highest possible power of $E$ before proceeding.
\end{remark}

At this point, the only part of the general process that we have yet to explain (except for the specific methods of generating target relations, covered in the next section) is given by the "\emph{Gr\"obner basis}" transitions from the figure, which test whether the target relation lies in the differential ideal generated by the base relations and the derivation $\partial$. One might expect differential Gr\"obner bases to be more useful in this case, since we are dealing with differential polynomials in one variable, rather than multivariate ordinary polynomials. However, differential Gr\"obner bases have the disadvantage of being infinite in certain cases, and therefore not suitable for machine computations.

Moreover, algorithms like \emph{Rosenfeld-Gr\"obner} \cite{Rosenfeld} for testing membership to a differential ideal only work for radical fields, which is not our case. We can work around this problem and at the same time increase the expected efficiency of our algorithms by the following observation:

\begin{remark}
Any member of the differential ideal generated by the differential polynomials $B_1(E)$ and $B_2(E)$ (from Theorem $3.1$) and the derivation $\partial$ can be written as:

\begin{equation}
\sum_k C_k \partial^k B_1(E) + \sum_k D_k \partial^k B_2(E),
\end{equation}

where $C_k$ and $D_k$ are expressions from the differential ideal generated by $E$ and the derivation $\partial$. So to control the elements of this form, it suffices to control the values of $\partial^k B_1(E)$ and $\partial^k B_2(E)$. In general, these values get more and more complicated as $k$ grows, but we are lucky to use a special derivation which obeys the property further described.

By Fermat's Little Theorem, one can see that given any power series $F$ and prime $q$:

\begin{equation}
\partial^q F \equiv_q \partial F,
\end{equation}

because $n^q \equiv_q n$ for each integer $n$. Therefore, the set $\{\partial^k E : k \geq 0\}$ contains at most $q$ elements, which we can enumerate as $X_0 := E, \ldots, X_{q-1} := \partial^{q-1} E$. Furthermore, the relation $\partial^{q} B_1(E) = \hat{0}$ must be equivalent to the relation $\partial B_1(E) = \hat{0}$. In fact, the relation $\partial^{q-1} B_1(E) = 0$ must be equivalent to the relation $B_1(E) = 0$, because $B_1(E)$ itself is of the form $\partial f(E)$ for some differential polynomial $f$ (this is true because $0$ must be among the residues $r_k$ and $s_k$ from equations $(54)$ and $(55)$, as previously explained). Thus it suffices to consider all expressions of the form in $(66)$ when $0 \leq k \leq q-2$. However, the factors $C_k$ and $D_k$ can be any multivariate polynomials in $X_0, \ldots, X_{q-1}$, and the expressions $\partial^k B_1(E)$, $\partial^k B_2(E)$ are fixed multivariate polynomials in the same "variables". Therefore, the differential ideal we are searching for is equivalent to the ideal generated by the following subset of the ring of polynomials in $X_0, \ldots, X_{q-1}$ over $\mathbb{Z}/q\mathbb{Z}$:

\begin{equation}
\beta := \{ \partial^k B_1(E) : 0 \leq k \leq q-2 \} \cup \{ \partial^k B_2(E) : 0 \leq k \leq q-2 \}
\end{equation}

Testing membership to this ideal is a much simpler problem: one just needs to generate a Gr\"obner basis starting from $\beta$, and then to apply a term-by-term reduction of the target relation by the elements of this basis. The target relation belongs to the ideal if and only if the polynomial obtained at the end of this reduction process is $\hat{0}$.

The exact same argument applies for $B_3(S)$ and $B_4(S)$ (the base relations given by Theorem $3.2$) instead of $B_1(E)$ and $B_2(E)$, except that we have to consider $0 \leq k \leq q-1$ rather than $0 \leq k \leq q-2$. This, of course, slows down the algorithms that search for divisor function congruences compared to those that search for partition congruences.
\end{remark}

\begin{implementation}
Our programs found at \cite{Maple1, Maple2} construct the initial set of differentiated base relations, denoted by $\beta$ in $(68)$, in an array called \emph{mybase}. We note that the variables $X_k$ from the description above are still remembered in the programs as $E[t, \ldots, t]$, respectively $S[t, \ldots, t]$. The algorithms then proceed to construct a Gr\"obner basis denoted by \emph{B} starting from \emph{mybase}, using the command \emph{Basis} from the MAPLE package \emph{Groebner}, and the ordering \emph{tdeg}$\{X_0, \ldots, X_{q-1}\}$ of the monomials (also known as \emph{grevlex} \cite{tdeg}). The form of this basis is automatically displayed when running the previously referenced MAPLE programs, after the definition of \emph{B}. Then the target relation is reduced by the Gr\"obner basis through the command \emph{NormalForm}. This command, as written in our programs, remembers the factors by which the elements of the basis are multiplied and then added together to obtain the target relation, in an array \emph{Q}. The $k^{\text{th}}$ entry of this array represents the factor corresponding to the $k^{\text{th}}$ element of the basis \emph{B}. The array \emph{Q} can be recovered (if the reader wants to reconstruct the machine computations) by introducing an additional MAPLE command $print(Q);$ (or only $Q;$). If the normal form obtained is $0$, then the reduction is successful and the proof of the desired congruence is complete.
\end{implementation}

\begin{remark}
Constructing a Gr\"obner basis from a sequence of base relations is the most time-consuming task executed by our algorithms. Because of this, we provide a few precomputed Gr\"obner bases in the MAPLE script available at \cite{Maple3}. These could be helpful if the reader wants to test our final computations independently, or to prove new types of congruences with our methods.
\end{remark}

\begin{example}
When $q = 5$, the equations from $(57)$ in the variable $E$ lead (after being differentiated $5 - 2 = 3$ times each) to the following Gr\"obner basis:
\begin{align}
\{&X_3 + \hat{2}X_2 + \hat{2}X_1, \\
&X_4 + \hat{3}X_2 + X_1, \\
&X_1^2 + \hat{3}X_0X_2 + \hat{2}X_0X_1\}
\end{align}

Since these polynomials are in the Gr\"obner basis found by our program, they should all equal $\hat{0}$ when $X_k = \partial^k E$ for $k \in \{0, 1, 2, 3, 4\}$.
\end{example}

\section{Congruences Proven by the Algorithms}

Here we prove the main results of this paper; while the algorithms deal with most of the hard computational work, our contribution will lie in finding adequate target relations. We will discuss:
\begin{enumerate}
\item Ramanujan-type congruences for partitions;
\item Ramanujan-type congruences for the divisor function;
\item Congruences for linear combinations of Ramanujan-type terms for the divisor function;
\item Congruences for weighted convolutions of the partition function;
\item Congruences for weighted convolutions of the divisor function.
\end{enumerate}

We note that in general, the congruences in $1$ are generalized by those in $4$ (in the numbering above), and the congruences in $2$ are generalized by those in $3$ and $5$. However, we will only study a subset of the very numerous congruences in $4$ and $5$ for computational reasons. We note that it is possible to test, with our methods, congruences for linear combinations as in $3$, but for partitions instead of the divisor function. The reason why these are not listed in a separate category of results is that for $q \in \{5, 7, 11\}$, all such congruences are trivial combinations of the congruences in $1$; this was checked by a numerical generation available at \cite{C1}.

Throughout the rest of this section, consider a prime $q > 3$.

\subsection{Ramanujan-Type Congruences}

This subsection deals with congruences of the forms in equation $(9)$, which we reiterate here:
\begin{align}
p^{*k}(qn + r) &\equiv_q 0, \\
\sigma^{*k}(qn + r) &\equiv_q 0,
\end{align}

where $k$ is an integer. As shown in Lemma $3.2$, all that matters for these congruences is the residue of $k$ modulo $q$. In fact, the same lemma gives us the differential algebraic formulation for these congruences (over $\mathbb{Z}/q\mathbb{Z}$):
\begin{align}
\left(\prod_{\substack{0 \leq s < q \\ s \neq r}} (\partial - \hat{s})\right) \left(P^{k}\right) &= \hat{0}, \\
\left(\prod_{\substack{0 \leq s < q \\ s \neq r}} (\partial - \hat{s})\right) \left(S^{k}\right) &= \hat{0}.
\end{align}

These are the target relations anticipated in Section 3: polynomial differential equations in $P$ and $S$, equivalent to the desired congruences. As explained in the remark after the figure from Section 3, the conversion of the target relation in $P$ to a target relation in $E$ is easy; in this case, it can be done by replacing $P$ with $E^{q-1}$ (since all the monomials in the LHS of $(74)$ have the same degree: $k$), then simplifying everything by the adequate power of $E$..

\begin{implementation}
In the programs available at \cite{Maple1} and \cite{Maple2}, the procedures \emph{buildEquation} (also used to build base relations), respectively \emph{buildTarget}, are used to derive target relations in $E$ and $S$, equivalent to equations $(74)$ and $(75)$, and remembered in each program as \emph{target}. Both procedures need some pre-computations about the residues modulo $q$ that should occur in the products from $(74)$ and $(75)$ (which in this case are all residues except for $r$). The program at \cite{Maple1} has direct definitions for the values of $k$ and $r$ (which can be modified by the reader), while the program at \cite{Maple2} prints all suitable pairs $(k, r)$  with $2 \leq k \leq r$ that lead to provable congruences. The reason for this difference in design is that the congruences for divisor functions are less studied, and thus less likely to be known prior to running the program. Moreover, most of the Ramanujan-type partition congruences are non-exceptional (i.e. $k \equiv_q 0, -1, -3$), so the program should not waste time trying to prove each of them.
\end{implementation}

The following theorem gathers all Ramanujan-type congruences with $k \not\equiv_q 0$ proven by our algorithms, both exceptional and non-exceptional:

\begin{theorem}[Ramanujan-Type Congruences]
The pairs $(k, r)$ with $0 < k < q$, $0 \leq r < q$ for which the Ramanujan-type congruence $p^{*k}(qn + r) \equiv_q 0$ holds ($\forall n$) are:
\begin{align}
\text{For $q = 5$:} \qquad &(1, 4), (2, 2), (2, 3), (2, 4), (4, 3), (4, 4)\\
\text{For $q = 7$:} \qquad &(1, 5), (4, 2), (4, 4), (4, 5), (4, 6), (6, 3), (6, 4), (6, 6) \\
\text{For $q = 11$:} \qquad &(1, 6), (3, 7), (5, 8), (7, 9), (8, 2), (8, 4), (8, 5), (8, 7), (8, 8), (8, 9), \\
&(10, 3), (10, 6), (10, 8), (10, 9), (10, 10)
\end{align}

The first pairs on each row lead to the three Ramanujan congruences \cite{Ramanujan}. In addition to the congruences above, our algorithms proved the congruence given by the pair $(13, 14)$ for $q = 17$, although this is not the only alleged congruence modulo $17$. We note that there are no exceptional congruences modulo $13$. Next, the pairs $(k, r)$ with $0 < k < q$, $0 \leq r < q$ for which the Ramanujan-type congruence $\sigma^{*k}(qn + r) \equiv_q 0$ holds ($\forall n$) are:
\begin{align}
&\text{For $q = 5$:} \qquad (2, 1), (3, 1), (3, 2), (4, 1), (4, 2), (4, 3)\\
&\text{For $q = 7$:} \qquad (2, 6), (4, 3), (4, 6), (5, 3)
\end{align}

The first pairs on each row lead to the two congruences proven by N. C. Bonciocat in \cite{Bonciocat}; the rest of them are new results to the best of the author's knowledge. One can verify numerically that there are no such congruences modulo $11$ or $13$.
\end{theorem}

\begin{proof}
Each of the partition congruences can be tested using the program at \cite{Maple1}, while each of the divisor function congruences can be tested with the program at \cite{Maple2}. In the former case, the reader has to specify the values of $q$, $k$, $r$ where indicated and then run the program to check whether the \emph{NormalForm} command prints the value $0$ (case in which the proof is successful). In the latter case, the reader only has to specify the value of $q$ and then run the program to display all valid pairs $(k, r)$. To reconstruct the proofs, one needs to add (or de-comment) the command \emph{print(Q);} right after the \emph{NormalForm} command, as described in the end of Section 3. This displays the factors with which the elements of the Gr\"obner basis (printed by the program after the definition of \emph{B}) are multiplied and then added together to obtained the alleged target relation. For exemplification, we further reconstruct the proof of the first Ramanujan congruence, $p(5n + 4) \equiv_5 0$. We have already proved this congruence via a similar method in Section $2$, but now the goal is to understand how our general algorithms work. While the proof we further provide might seem shorter than the previous one, we must remark that it skips much of the computational work (such as computing the Gr\"obner basis and the target relation).

In Example $3.4$, we displayed the Gr\"obner basis for partition congruences, generated by our program when $q = 5$ (recall $X_k := \partial^k E$). The target relation for the first Ramanujan congruence, obtained after expanding equation $(74)$ with $q = 5$, $k = 1$ and $r = 4$, is (as generated by the procedure \emph{BuildEquation} after the definition of \emph{target}):

\begin{equation}
\begin{split}
X_0^3(X_1 + \hat{4}X_2 + X_3 + \hat{4}X_4) + X_0^2(\hat{2}X_1^2 + \hat{4}X_1X_2 + \hat{3}X_1X_3 + X_2^2) \q+ \\
X_0(X_1^3 + \hat{4}X_1^2X_2) + \hat{4}X_1^4 = \hat{0}
\end{split}
\end{equation}

The factors generated in the array \emph{Q} of the program (together with the form of the basis \emph{B}) indicate that the expression in the LHS above can be decomposed as:

\begin{equation}
\begin{split}
&\left(X_0^3 + \hat{3}X_0^2 X_1\right)\left(X_3 + \hat{2}X_2 + \hat{2}X_1\right) + \\
&\hat{4} X_0^3 \left(X_4 + \hat{3}X_2 + X_1\right) + \\
&\left(\hat{4}X_1^2 + \hat{2}X_0X_2 + \hat{3}X_0X_1 \right)\left(X_1^2 + \hat{3}X_0X_2 + \hat{2}X_0X_1\right),
\end{split}
\end{equation}

which can also be checked manually. The second factor in each product above is $\hat{0}$ according to $(69)$, $(70)$, and $(71)$ (these are precisely the elements of the Gr\"obner basis), so our proof is complete.
\end{proof}

\begin{remark}
A numerical generation shows that for $q \in \{5, 7, 11\}$, the pairs displayed in Theorem $4.1$ are the only ones that lead to valid congruences.
\end{remark}

\subsection{Linear Combinations of Ramanujan-Type Terms}

Consider, now, congruences of the following form over $\mathbb{Z}/q\mathbb{Z}$:
\begin{align}
\left( c_1p + c_2p^{*2} + \ldots + c_{q-1}p^{*q-1} \right)(qn + r) &= \hat{0} \q\q(\forall n), \\
\left( c_1\sigma + c_2\sigma^{*2} + \ldots + c_{q-1}\sigma^{*q-1} \right)(qn + r) &= \hat{0} \q\q(\forall n),
\end{align}

where each $c_k \in \mathbb{Z}/q\mathbb{Z}$. These are perhaps the most immediate generalizations of Ramanujan-type congruences modulo $q$, since they just take linear combinations of the terms $p^{*k}(qn + r)$, $\sigma^{*k}(qn + r)$. We restricted the convolution exponents here to $k \in \{1, \ldots, q-1\}$ because there would be too many trivial congruences if we also allowed $k \equiv_q 0$. Also, there is some sort of repetition of these congruences modulo $q$, in the sense that one can add any fixed multiple of $q$ to all convolution exponents while preserving the congruence (this follows from the computations below and the observations from Section $3$). However, our methods can in principle be used to prove such congruences for any range of the convolution exponent $k \in \mathbb{Z}$.

\begin{remark}
Clearly, all $(q-1)$-tuples that satisfy relation $(84)$ for a fixed residue $r$, respectively all $(q-1)$-tuples that satisfy relation $(85)$ for a fixed $r$, form vector spaces over $\mathbb{Z}/q\mathbb{Z}$. Our goal is therefore to determine bases for these vector spaces (in order to avoid listing thousands of congruences with no pattern).

As priorly mentioned, these relations generalize the Ramanujan-type congruences modulo $q$. Indeed, if all coefficients $c_k$ were $\hat{0}$ except for one, one would obtain a congruence of the form $p^{*k}(qn + r) \equiv_q 0$ or $\sigma^{*k}(qn + r) \equiv_q 0$. Conversely, starting from a set of Ramanujan-type congruences for partitions, respectively divisor functions (as in the previous subsection), any linear combination of them will satisfy relation $(84)$, respectively $(85)$. We say that these linear combinations lead to \emph{non-exceptional} congruences; the big question is whether there are any \emph{exceptional} ones.

Unfortunately, a numerical generation available at \cite{C1} (writen in C++ this time) shows that there are no exceptional congruences of this form for partitions (i.e., as in $(84)$), at least not with $q \in \{5, 7, 11\}$. On the other hand, there are lots of them for divisor functions. We therefore only study the congruences of the form in $(85)$ for the rest of this subsection.
\end{remark}

We start by noticing that any congruence of the form in $(85)$ can be rephrased as:

\begin{equation}
\left(\prod_{\substack{0 \leq s < q \\ s \neq r}} (n - \hat{s})\right) \left( c_1\sigma + c_2\sigma^{*2} + \ldots + c_{q-1}\sigma^{*q-1} \right)(n) = \hat{0}, \q\q(\forall n),
\end{equation}

by the same judgment as in the proof of Lemma $3.2$. In the language of Differential Algebra over $\mathbb{Z}/q\mathbb{Z}$, the relation above translates to:

\begin{equation}
\left(\prod_{\substack{0 \leq s < q \\ s \neq r}} (\partial - \hat{s})\right) \left(c_1S + c_2S^2 + \ldots + c_{q-1}S^{q-1}\right) = \hat{0},
\end{equation}

which is a perfectly good target relation for our purposes: a polynomial differential equation in $S$. As a side note, multiplying this relation by $S^{kq}$ will just add $kq$ to each of the (convolution) exponents, by the same reasoning as in Lemma $3.2$.

\begin{implementation}
In the program available at \cite{Maple2}, the procedure \emph{testCombo} deals with constructing the target relation corresponding to the expansion of $(87)$, as well as with testing the membership of this target relation to the differential ideal generated by the base relations and the derivation $\partial$. The \emph{testCombo} procedure takes two arguments: the residue $r$ and the sequence $(c_1, \ldots, c_{q-1})$, and prints $0$ if the proof attempt was successful. To recover the form of the target relation and the array \emph{Q} of factors used in the reduction by the Gr\"obner basis, just add the commands \emph{print(target);} and \emph{print(Q);} in the implementation of \emph{testCombo}.
\end{implementation}

\begin{theorem}[Linear Combination Congruences]
For $q \in \{5, 7, 11\}$ and $0 \leq r < q$, the space:

\begin{equation}
V_{q, r} := \left\{ (c_1, \ldots, c_{q-1}) \in (\mathbb{Z}/q\mathbb{Z})^q : \left( c_1\sigma + c_2\sigma^{*2} + \ldots + c_{q-1}\sigma^{*q-1} \right)(qn + r) = \hat{0} \q\q(\forall n)\right\}
\end{equation}

has the basis indicated in the third row of the first table from the introduction.

All the exceptional basis congruences are shown in bold, since they are new results. As mentioned in the introduction, the bases of the spaces $V_{11, r}$ are printed by omitting the last three coefficients, which are understood to be $\hat{0}$. All bases are shown in reduced row echelon form.
\end{theorem}

\begin{proof}
The proof of this statement has three steps. First, one must generate a complete set of congruences of the form in $(85)$ (while ignoring perhaps some repetitions by scaling), which verify numerically up to large values. This is done in a C++ program available at \cite{C1}. Then, for fixed values of $q$ and $r$, one must construct a basis for the span of the generated vectors $(c_1, \ldots, c_{q-1})$, which is allegedly $V_{q, r}$. This can be done through a short MAPLE program available at \cite{Maple3}. Finally, one must test that our main program from \cite{Maple2} can prove each of the congruences from the alleged basis, following the steps from the implementation notes above. The reader can check that this approach is successful for all the basis vectors that verify numerically.
\end{proof}

\begin{remark}
The table from Section $1$ only provides a basis for each vector space studied, but the complete spaces of congruences involved are much more numerous and, in some cases, more appealing than the basis congruences. We further provide a simple example:
\end{remark}

\begin{example}
Consider the vector space $V_{11, 2}$; one of its bases, according to the table in the introduction, is:

\begin{equation}
\begin{split}
\{&(1, 8, 0, 0, 0, 2, 10, 0, 0, 0),\\
&(0, 0, 1, 0, 0, 5, 3, 0, 0, 0),\\
&(0, 0, 0, 1, 0, 7, 2, 0, 0, 0),\\
&(0, 0, 0, 0, 1, 4, 9, 0, 0, 0)\}
\end{split}
\end{equation}

By adding the last two vectors in this basis, one obtains the simpler $(q-1)$-tuple:

\begin{equation}
(0, 0, 0, 1, 1, 0, 0, 0, 0, 0)
\end{equation}

This corresponds to the congruence anticipated in $(15)$, which can also be written as:

\begin{equation}
\sigma^{*4}(n) \equiv_{11} -\sigma^{*5}(n), \quad\quad \forall n \equiv_{11} 2
\end{equation}

\end{example}

\subsection{Weighted Convolutions of Partitions and Divisor Functions}

Recall the definition of weighted convolutions from $(17)$. After some thought, one can see that any polynomial differential equation generated by some power series $F(X) = \sum_{n \geq 0} f(n)X^n$ and the derivation $\partial$, over the field $\mathbb{Z}/q\mathbb{Z}$, can be written as:

\begin{equation}
\sum_{k \geq 0} f^{*k}_{w_k} \equiv_q 0,
\end{equation}

where each $w_k$ is a multivariate polynomial in $k$ variables. Indeed, this can be achieved by grouping all differential monomials of degree $k$ into a single weighted convolution $f^{*k}_{w_k}$; the exponents in the monomials of $w_k$ will correspond to the exponents of $\partial$ in the differential monomials of the initial equation. For example, $F\partial^2 F + \hat{3}(\partial F)^2 = \hat{0}$ amounts to $f^{*2}_{w} \equiv_q 0$, where $w(a, b) = a^0b^2 + 3ab$.

All the congruences with which our algorithms operate are, thus, translatable into the form of $(92)$, where $f$ is the partition function or the divisor function. However, not every such congruence can be considered a significant result, because the expressions of the polynomials $w_k$ can get very lengthy.

The goal is, therefore, to find particular cases of equations like $(92)$ with reasonably simple forms. The Ramanujan-type congruences and the linear combinations of Ramanujan-type terms studied before are all examples of such particular cases. Another idea originated from this perspective is to consider individual terms $p^{*k}_w$ and $\sigma^{*k}_w$, for small values of $k$. Can one characterize the polynomials $w$ for which these terms cancel out modulo $q$? Indeed, such relations would generalize all Ramanujan-type congruences studied in Subsection $4.1$, because the differential polynomials in $(74)$ and $(75)$ have a uniform degree.

We further study the subproblem of the congruences $p^{*2}_{w} \equiv_q 0$ and $\sigma^{*2}_{w} \equiv_q 0$, i.e. we set $k = 2$. As explained above, these are equivalent to the following differential equations over $\mathbb{Z}/q\mathbb{Z}$:
\begin{align}
\sum_{j, k \geq 0} \left[a^jb^k\right](w(a, b)) \cdot \partial^j P \q\partial^k P &= \hat{0} \\
\sum_{j, k \geq 0} \left[a^jb^k\right](w(a, b)) \cdot \partial^j S \q\partial^k S &= \hat{0}
\end{align}

Here, $\left[a^j b^k\right](w(a, b))$ denotes the coefficient of the monomial $a^j b^k$ in the polynomial $w(a, b)$. These are the target relations that our algorithms will attempt to prove, with the mention that $(93)$ must be first converted to a polynomial differential equation in $E$. Since all monomials have the same degree: 2, it suffices to replace $P$ by $E^{q-1}$, as explained in the end of Section $3$, and then to simplify by an adequate power of $E$.

\begin{remark}
From the nature of $(93)$ and $(94)$, one can see that certain simplifications to the polynomial $w$ are possible: we can, for example, replace any monomial $a^j b^k$ with $a^k b^j$; in general, any permutation of the exponents within a monomial is allowed. \textbf{The convention that we will follow is that the exponents of these monomials should be listed in decreasing order.}

Another observation is that each congruence of the form in $(93)$ or $(94)$ can be differentiated by $\partial$ to obtain another such congruence; only the polynomial $w$ will change. Since relations like $(74)$ and $(75)$ are abundant, this observation yields a wide range of valid weighted-convolution congruences. For simplicity, we will focus on discovering the valid polynomials $w$ that consist of a reduced number of monomials.
\end{remark}

\begin{implementation}
In the programs available at \cite{Maple1} and \cite{Maple2}, the procedure \emph{testWeight} deals with constructing the polynomial differential equations in $(93)$ (translated into an equation in $E$), respectively $(94)$, as well as with testing whether these target relations belong to the ideal generated by the base relations and the derivation $\partial$. The procedure \emph{testWeight} takes one argument: a sequence of monomials in the variables $a$ and $b$ whose sum yields the polynomial $w(a, b)$ (intuitively, one should just replace each '+' sign in $w(a, b)$ with a comma; the reader should be careful to separate $a$ and $b$ by a space or star when necessary). As before, the value of $q$ is specified at the beginning of the programs. Ultimately, \emph{testWeight} prints the value $0$ if the proof of the target relation was successful; to recover the target equation and the array of factors used in the reduction by the Gr\"obner basis, just de-comment the instructions \emph{print(eq);} and \emph{print(Q);} from the implementation of \emph{testWeight}. Running our algorithms leads to the following theorem:
\end{implementation}

\begin{theorem}[Weighted Convolution Congruences]
When $q \in \{5, 7\}$, one has:

\begin{equation}
p^{*2}_w \equiv_5 0, \qquad \text{respectively} \qquad \sigma^{*2}_w \equiv_7 0,
\end{equation}

for all corresponding polynomials $w(a, b)$ shown in the second table from the introduction. In each case, those are the only valid nonzero polynomials $w$ over $\mathbb{Z}/q\mathbb{Z}$ consisting of at most $q-2$ monomials, up to scaling and to the convention that exponents are listed in decreasing order (see the remark after $(94)$). The prime $q$ is restricted to $\{5, 7\}$ for reasons of computational efficiency. All of the resulting congruences are new, to the best of the author's knowledge.
\end{theorem}

\begin{proof}
The program available at \cite{C2} generates all the nonzero polynomials $w$ that lead to valid congruences of the forms in $(95)$ (verified up to large values), up to scaling by a factor, such that $w$ consists of at most $q - 2$ monomials. This upper bound seems to provide the right equilibrium between the simplicity and the number of results obtained, but it can of course be modified by the reader where indicated in the program.

Then, each allegedly valid polynomial $w$ can be plugged into the procedure \emph{testWeight} from our programs available at \cite{Maple1, Maple2}, as detailed in the previous implementation notes. Luckily, all the alleged congruences obtained for $q \in \{5, 7\}$ are successfully proven by our algorithms.
\end{proof}

\begin{remark}
As in the previous subsection, one can take linear combinations of the valid polynomials $w$ to obtain new valid polynomials (this is clear from the definition of weighted convolution):
\end{remark}

\begin{example}
Consider the polynomials $w$ that lead to valid congruences of the form $p^{*2}_{w} \equiv_7 0$, shown in the second table from the introduction:
\begin{align}
w_1(a, b) &= a^6 + a^5 + 6a^4 + 6a^2 + 2ab \\
w_2(a, b) &= a^6 + 3a^4 + 2a^3 + 6ab + a
\end{align}

Subtracting the second polynomial from the first uncovers another valid polynomial $w$, which has degree $5$ (but more monomials):

\begin{equation}
w(a, b) = a^5 + 3a^4 + 5a^3 + 6a^2 + 3ab + 6a
\end{equation}
\end{example}

\section{Related Congruences \& Conjectures for Divisor Functions}

While developing the proofs and the algorithmic approach presented in the previous sections, we came across several miscellaneous results that are worth mentioning. First, we need some notations:

\begin{notation} Given a positive integer $n$, let $\sigma_{odd}(n)$ and $\sigma_{even}(n)$ denote the sum of all odd, respectively all even positive divisors of $n$ (with $\sigma_{odd}(n) := \sigma_{even}(n) := 0$ for $n \leq 0$). Also, let $\{\sigma_{inv}(n)\}_{n \geq 0}$ denote the convolution inverse of the sequence $\{\sigma(n+1)\}_{n \geq 0}$, i.e.:
\end{notation}

\begin{equation}
\sum_{k = 0}^{n} \sigma(k+1)\sigma_{inv}(n-k) = \begin{cases}
1, & \text{ if } n = 0, \\
0, & \text{ else.}
\end{cases}
\end{equation}

We can now state our results:

\begin{proposition}
If $q$ is an odd prime, then the following two congruences are equivalent:
\begin{align}
\sum_{k = 1}^{q-1} c_k\sigma^{*k}(qn + r) &\equiv_q 0  \q\q(\forall n) \\
\sum_{k = 1}^{q-1} \left(\frac{q+1}{2}\right)^k c_k{\sigma_{even}}^{*k}(qn + 2r) &\equiv_q 0  \q\q(\forall n)
\end{align}
\end{proposition}

\begin{proof}
This proposition is a direct consequence of the following elementary formula:

\begin{equation}
\sigma_{even}(n) = \begin{cases}
2\sigma\left(\frac{n}{2}\right), & n \equiv_2 0 \\
0, & n \equiv_2 1
\end{cases}
\end{equation}

Indeed, the even divisors of an even integer are just twice the divisors of half of that integer, and an odd integer has no even divisors. More generally, for any $k \geq 1$:

\begin{equation}
{\sigma_{even}}^{*k}(n) = \begin{cases}
2\sigma^{*k}\left(\frac{n}{2}\right), & n \equiv_2 0 \\
0, & n \equiv_2 1
\end{cases}
\end{equation}

This is because in any convolution sum ${\sigma_{even}}^{*k}(n) = \sum_{n_1 + \ldots + n_k = n} \sigma_{even}(n_1) \ldots \sigma_{even}(n_k)$, the only way how a term can be nonzero is if $n_1, \ldots, n_k$ are all even (and in consequence $n$ is even too). Therefore, $\sigma_{even}^{*k}(n) = 0$ when $n$ is odd, and when $n$ is even we can write:

\begin{equation}
\begin{split}
{\sigma_{even}}^{*k}(n) &= \sum_{2n_1 + \ldots + 2n_k = n} \sigma_{even}(2n_1) \ldots \sigma_{even}(2n_k) \\
&= \sum_{n_1 + \ldots + n_k = \frac{n}{2}} 2^k \sigma(n_1) \ldots \sigma(n_k) \\
&= 2^k \sigma^{*k}\left(\frac{n}{2}\right)
\end{split}
\end{equation}

Finally, using $(103)$, we can rewrite the LHS of $(101)$ as:
\begin{equation}
\sum_{k = 1}^{q-1} \left(\frac{q+1}{2}\right)^k c_k{\sigma_{even}}^{*k}(qn + 2r) = \sum_{k = 1}^{q-1} (q+1)^k c_k\sigma^{*k}\left(\frac{qn}{2} + r\right),
\end{equation}

when $n$ is even. Therefore, the statement that $(101)$ holds for all even $n$ is equivalent to the statement that $(100)$ holds for all $n$. However, $(101)$ always holds when $n$ is odd because the LHS is just $0$ in that case, so $(100)$ and $(101)$ are completely equivalent.

\end{proof}

As a consequence, all of our congruences concerning linear combinations of divisor function convolutions can be uniformly translated into congruences for the even divisor function, by doubling the value of $r$ modulo $q$ and scaling the coefficients accordingly. In particular, from the convolution congruences $\sigma^{*2}(5n + 1) \equiv_q 0$ and $\sigma^{*2}(7n + 6) \equiv_q 0$, established in the previous section as well as in N. C. Bonciocat's paper \cite{Bonciocat}, we can infer two of Bonciocat's conjectures from the same paper:

\begin{corollary}[Bonciocat's First Two Conjectures]
For any $n \geq 0$, one has:
\begin{align}
{\sigma_{even}}^{*2}(5n+2) &\equiv_5 0 \\
{\sigma_{even}}^{*2}(7n+5) &\equiv_7 0
\end{align}
\end{corollary}

The third and last of N. C. Bonciocat's conjectures is the following result, which unfortunately cannot be generalized (at least not immediately) via our techniques.

\begin{proposition}[Bonciocat's Third Conjecture]
For any $n \geq 0$, one has:
\begin{equation}
{\sigma_{odd}}^{*2}(5n+1) \equiv_5 0
\end{equation}
\end{proposition}

\begin{proof}
While this result is not directly related to the methods presented in this paper, we give it a short proof for the sake of completeness, and to illustrate a different way of obtaining divisor function congruences, starting from certain convolution identities. We start from the following result of D. Kim \cite{Kim, Kim2, Kim3}:

\begin{equation}
{\sigma_{odd}}^{*2}(n) = \frac{1}{24}\left(11\sigma_3(n) - \sigma_3(2n) - 2\sigma_{odd}(n)\right)
\end{equation}

Here, $\sigma_3(n)$ denotes the sum of the cubes of all positive divisors of $n$; the referenced papers of Kim gather many more identities of this form. Taking $(109)$ modulo 5, we find that:

\begin{equation}
{\sigma_{odd}}^{*2}(n) \equiv_5 2\sigma_{odd}(n) + \left(\sigma_3(2n) - \sigma_3(n)\right)
\end{equation}

Write $n = 2^p m$ such that $m$ is odd. Then we can rewrite $(110)$ as:

\begin{equation}
\begin{split}
{\sigma_{odd}}^{*2}(n) &\equiv_5 2\sigma(m) + \sum_{\substack{d \divides 2^{p+1}m \\ d \not\divides 2^pm}} d^3 \\
&= 2\sigma(m) + \sum_{d \divides m} (2^{p+1}d)^3 \\
&= 2\sum_{d \divides m} \left(4(2^p d)^3 + d\right) \\
&= \sum_{d \divides m} \left(4(2^p d)^3 + d\right) + \sum_{d \divides m} \left(4\left(2^p \frac{m}{d}\right)^3 + \frac{m}{d} \right) \\
&= 2\sum_{d \divides m} \left(4\left(2^p\frac{m}{d}\right)^3 + d\right)
\end{split}
\end{equation}

But $2^p m$ is precisely $n$, so for $n \equiv_5 1$, $2^p\frac{m}{d} \equiv_5 d^{-1}$ (note that $d \not\equiv_5 0$ because $d \divides m \divides n \equiv_5 1$). Thus each term of the sum above is congruent modulo 5 to $d - d^{-3} = d(1 - d^{-4}) \equiv_5 0$, since $d^4 \equiv_5 1$ by Fermat's Little Theorem. Thus the whole sum is divisible by $5$ when $n \equiv_5 1$, which completes our proof.
\end{proof}

\begin{remark}
While it is possible to obtain congruences like the one in $(108)$ starting from identities of divisor function convolutions, the computations generally get very complicated because of the higher-order divisor functions $\sigma_k$; there is no guarantee that such an approach would be successful with more general classes of congruences.
\end{remark}

The next proposition extends a few divisor function congruences gathered in Theorem $4.1$:

\begin{proposition}
For any $n \geq 0$, one has:
\begin{equation}
\sigma_{inv}(n) \equiv_5 0, \text{ for } 1 < n \not\equiv_5 0
\end{equation}
\end{proposition}

\begin{proof}
We will use a partial result reached within the proof of Theorem $2.1$ (see relation $(25)$):

\begin{equation}
T' \equiv_5 2,
\end{equation}

where $T := \frac{P}{P'}$. We can write by this definition that:

\begin{equation}
T^{-1}(X) := \frac{P'}{P}(X) = \sum_{n = 0}^\infty \sigma(n+1)X^n
\end{equation}

The second equality above holds because $\frac{XP'(X)}{P(X)} = \frac{\partial P}{P}(X) = S(X) = \sum_{n = 1}^\infty \sigma(n)X^n$. Now by the way we defined $\sigma_{inv}$ in $(99)$, we see that:

\begin{equation}
T(X) = \sum_{n = 0}^\infty \sigma_{inv}(n)X^n
\end{equation}

However, $(113)$ implies that $T$ can only contain (modulo 5) powers of the form $X^n$ where $n \equiv_5 0$, except for $n = 1$. Therefore, one has $\sigma_{inv} (n) \equiv_5 0$ for all $1 < n \not\equiv_5 0$.
\end{proof}

\begin{remark}
The congruences $\sigma^{*4}(5n + r) \equiv_5 0$ for $r \equiv_5 1, 2, 3$, established in Section $4$, can be shifted by a multiple of $5$ in the convolution exponent, due to Lemma $3.2$:

\begin{equation}
\sigma^{*-1}(5n + r) \equiv_5 0, \quad \text{for } r \in \{1, 2, 3\}
\end{equation}

The relation above holds true provided that we understand the function $\sigma^{*-1}$ properly. Indeed, $\sigma$ has no convolution inverse by the standard definition because $\sigma(0) = 0$, but we can resolve this by allowing $\sigma^{*-1}(-1) = 1$. This corresponds to what happens when we try to invert the power series $S(X)$: we must use a negative power $X^{-1}$ in the inverse. In this sense, one can easily see that $\sigma_{inv}(n) = \sigma^{*-1}(n-1)$ for all $n$, so equation $(116)$ can be rephrased as:

\begin{equation}
\sigma_{inv}(5n + r) \equiv_5 0, \quad \text{for } r \in \{2, 3, 4\}
\end{equation}

This is indeed a particular case of the more general equation $(112)$. Hence, Proposition $5.3$ is just an extension of three previously proven divisor function congruences, for the particular convolution exponent $k = -1$.

\end{remark}

We end with a congruence for an interesting summation that resembles a (discrete) convolution:

\begin{proposition}
For any positive integer $n$, one has:
\begin{equation}
\sum_{a+b = n} \sigma(ab) \equiv_5 2(n^2 - 1),
\end{equation}

where it is understood that $a$ and $b$ are integers (although it is sufficient if they range over the positive integers, since $\sigma(m) = 0$ by convention for $m \leq 0$). In particular, the residues modulo $5$ of the LHS are periodic with period $5$ (in terms of $n \geq 1$): $0, 1, 1, 0, 3, 0, 1, 1, 0, 3, \ldots$
\end{proposition}

\begin{proof}
We start by proving the following lemma:

\begin{lemma}
For any positive integer $n$, one has:
\begin{equation}
\sigma(a)\sigma(b) = \sum_{d \divides \gcd(a, b)} d \q\sigma\left( \frac{ab}{d^2} \right)
\end{equation}
where the sum is taken over all positive divisors of $n$.
\end{lemma}

To prove the lemma, induct on the value of $g := \gcd(a, b)$. When $g = 1$, i.e. $a$ and $b$ are relatively prime, we know that $\sigma(a)\sigma(b) = \sigma(ab)$ since $\sigma$ is multiplicative. Next, for $g > 1$, pick a maximal prime power $p^k$ dividing $g$ and write $g = p^k g'$, $a = p^k a'$, $b = p^k b'$. Assume WLOG that $b'$ is not divisible by $p$. Then $\gcd(p^{q}a', b') = \gcd(a', b') = g'$, for any $q \geq 0$. Using the induction hypothesis on $g'$, we see that:

\begin{equation}
\begin{split}
\sum_{d \divides p^k g'} p^k d \q\sigma\left(\frac{p^{2k} a'b'}{d^2}\right)
&= \sum_{i = 0}^k p^i \sum_{d \divides g'} d\q\sigma\left(\frac{p^{2(k-i)} a'b'}{d^2}\right) \\
&= \sum_{i = 0}^k p^{k-i} \sum_{d \divides g'} d\q\sigma\left(\frac{\left(p^{2i} a'\right)b'}{d^2}\right) \\
&= \sum_{i = 0}^k p^{k-i} \sigma\left(p^{2i}a'\right)\sigma(b')
\end{split}
\end{equation}

On the other hand, the LHS of $(119)$ can in this context be written as:

\begin{equation}
\begin{split}
\sigma(a)\sigma(b) &= \sigma\left(p^k a'\right)\sigma\left(p^k b'\right) \\
&= \sigma\left(p^k a'\right)\sigma\left(p^k\right) \sigma(b')
\end{split}
\end{equation}

Thus to complete the induction, it suffices to show that:

\begin{equation}
\sum_{i = 0}^k p^{k-i} \sigma\left(p^{2i} a'\right) = \sigma\left(p^k a'\right)\sigma\left(p^k\right)
\end{equation}

By writing $a' = p^j c$ with $j$ maximal, one can take away a factor of $\sigma(c)$ from both sides of the relation above (since $\sigma$ is multiplicative) to get the equivalent statement:

\begin{equation}
\begin{split}
\sum_{i = 0}^k p^{k-i} \sigma\left(p^{2i+j}\right) = \sigma\left(p^{k+j}\right)\sigma\left(p^k\right) \quad &\iff \quad \sum_{i = 0}^k p^{k-i} \frac{p^{2i+j+1}-1}{p-1} = \frac{p^{k+j+1} - 1}{p-1} \frac{p^{k+1} - 1}{p-1} \\
&\iff \quad \sum_{i = 0}^k \left(p^{k+i+j+1}-p^{k-i}\right) = \left(p^{k+j+1} - 1\right) \sum_{i = 0}^k p^i
\end{split}
\end{equation}

The last relation is easily checked, so the proof of our lemma is complete.

Now denote the LHS of equation $(118)$ by $f(n)$. Using the lemma, we can see that:

\begin{equation}
\begin{split}
\sigma^{*2}(n) &= \sum_{a + b = n} \sigma(a)\sigma(b) \\
&= \sum_{a + b = n} \q\sum_{d \divides \gcd(a, b)} d\q \sigma\left(\frac{ab}{d^2}\right) \\
&= \sum_{d \divides n} d \sum_{\substack{a + b = n, \\ d \divides \gcd(a, b)}} \sigma\left(\frac{ab}{d^2}\right) \\
&= \sum_{d \divides n} d \sum_{a + b = \frac{n}{d}} \sigma(ab) \\
&= \sum_{d \divides n} d \q f\left(\frac{n}{d}\right)
\end{split}
\end{equation}

By the \emph{M\"obius Inversion Formula}, the identity  above uniquely determines $f$ on the positive integers. In particular, it uniquely determines the residues of $f$ modulo $5$.

This is where our previous congruences become helpful. We turn back to equations $(113)$ and $(114)$ to infer that:

\begin{equation}
2 \equiv_5 \left(\frac{X}{S(X)}\right)' = \frac{S(X) - XS'(X)}{S(X)^2}
\end{equation}

Thus:
\begin{equation}
S^2 \equiv_5 2(\partial S - S)
\end{equation}

In fact, we have reached this differential equation before - or at least our algorithms have; see the second relation in $(65)$. By identifying coefficients, we get that:

\begin{equation}
\sigma^{*2}(n) \equiv_5 2n\sigma(n) - 2\sigma(n) = \sum_{d \divides n} d \cdot 2\left(\left(\frac{n}{d}\right)^2 - 1\right)
\end{equation}

As a side note, the relation above can also be derived from an identity of Besge \cite{Hahn}. From $(124)$ and $(127)$ we deduce (using the remark about the \emph{M\"obius Inversion Formula}) that:

\begin{equation}
f(n) \equiv_5 2(n^2 - 1),
\end{equation}

for all positive integers $n$, which completes our proof.

\end{proof}

\section{Remarks}

\begin{enumerate}
\item Winquist's identity \cite{Winquist} might be useful as a third source of base relations, when one considers congruences modulo bigger primes. The conversion of this identity to a differential equation (modulo a prime $q$) would follow the same principles as that of Euler's Pentagonal identity and Jacobi's identity.

\item As we noted in Section $3$, the differential equations that we can obtain in terms of the power series $E$ are usually shorter and simpler than the differential equations in terms of $S$. Therefore, if all the Gr\"obner bases computations were done in terms of $E$ instead of $S$, our algorithms might become slightly more efficient. However, this would require a conversion of the target divisor function congruences to differential equations in $E$ (by the relation $S = -\frac{\partial E}{E}$), which would significantly increase the complexity of the target equations. In particular, to rewrite the term $\partial^k S$ in terms of the power series $E$, one could use the following recursive step:
\begin{equation}
\begin{split}
\partial^k S
&= \partial ( \partial^{k-1} S) \\
&= \partial \left( \frac{f_{k-1}(E)}{E^k} \right) \\
&= \frac{E \partial f_{k-1}(E) - kf_{k-1}(E)\partial E}{E^{k+1}},
\end{split}
\end{equation}

where $f_{k-1}(E)$ is defined such that $\partial^k S = \frac{f_k(E)}{E^{k+1}}$. In fact, the relation above gives the recursive identity $f_k(E) = E \partial f_{k-1}(E) - kf_{k-1}(E) \partial E$. Since this recursion starts from $\partial^0 S = S = \frac{-\partial E}{E}$, whence $f_0(E) = -\partial E$, one can see that $f_k$ remains a differential polynomial for all $k \geq 0$.

In order to convert an entire polynomial differential equation in $S$ to one in terms of $E$, one would need to apply the transformation in $(129)$ to each factor of the form $\partial^k S$, and then multiply everything by an adequate power of $E$. This would become computationally problematic when one deals with monomials like $\left(\partial^k S\right)^{l}$, since the converted expression would have a (potentially huge) denominator of $E^{(k+1)l}$. Moreover, the multiplication of such expressions would lead to very lengthy polynomials. Nevertheless, the reader is encouraged to try to improve this alternative approach for proving divisor function congruences.

\item In Subsection $4.3$, we did not exhaust all possible congruences of the forms $p_w^{*k} \equiv_q 0$ and $\sigma_w^{*k} \equiv_q 0$, since we only considered the case $k = 2$ for reasons of efficiency. The computationally difficult part is generating such congruences that verify numerically up to large values; if the reader could find a quicker way to generate these congruences numerically, then their algorithmic proofs based on our differential algebraic methods should not take too long.

\item It would be interesting to study if the divisor function congruences proven in this paper, combined with the properties of multiplicative functions, can lead to more general results than the one from Proposition $5.3$. A common approach to such results could be based on Dirichlet series and generalizations of Lemma $5.1$ (e.g., expressing the product $\sigma(a)\sigma(b)\sigma(c)$).
\end{enumerate}
\pagebreak


\begin{thebibliography}{30}

\bibitem {Ramanujan}
Berndt, B. C.
\emph{Ramanujan's congruences for the partition function modulo 5, 7, and 11}.
Int. J. Number Theory 3 (2007), no. 3, 349–354. 

\bibitem{Ono}
Eichhorn, D.; Ono, K.
\emph{Congruences for partition functions}. Analytic number theory, Vol. 1 (1995), 309–321.

\bibitem{Abramo}
Abramowitz, M.; Stegun, I. A.
\emph{Handbook of mathematical functions with formulas, graphs, and mathematical tables}. Dover Publications, Inc., New York (1992).

\bibitem {Pentagonal}
Liu, J. C.
\emph{Some finite generalizations of Euler's pentagonal number theorem}. 
Czechoslovak Math. J. 67(142) (2017), no. 2, 525–531. 

\bibitem{Jacobi}
Srivastava, H. M.; Chaudhary, M. P.; Chaudhary, S.
\emph{Some theta-function identities related to Jacobi's triple-product identity}. 
Eur. J. Pure Appl. Math. 11 (2018), no. 1, 1–9.

\bibitem {Winquist}
Nupet, C.; Kongsiriwong, S.
\emph{Two simple proofs of Winquist's identity}. 
Electron. J. Combin. 17 (2010), no. 1, Research Paper 116.

\bibitem {JacWin}
Berndt, B. C.; Gugg, C.; Kim, S.
\emph{Ramanujan's elementary method in partition congruences}. Partitions, q-series, and modular forms, 13–21, 
Dev. Math., 23, Springer, New York (2012).

\bibitem {ShortProof}
Hirschhorn, M. D.
\emph{A short and simple proof of Ramanujan's mod11 partition congruence}.
J. Number Theory 139 (2014), 205–209.

\bibitem{Algorithm1}
Gnang, E.; Zeilberger, D.
\emph{Generalizing and implementing Michael Hirschhorn's amazing algorithm for proving Ramanujan-type conjectures}. 
Gac. R. Soc. Mat. Esp. 17 (2014), no. 1, 129–138. 

\bibitem {Ono2}
Ono, K.
\emph{Distribution of the partition function modulo m}. 
Ann. of Math. (2) 151 (2000), no. 1, 293–307. 

\bibitem{Ono3}
Ahlgren, S.; Ono, K.
\emph{Congruence properties for the partition function}. 
Proc. Natl. Acad. Sci. USA 98 (2001), no. 23, 12882–12884. 

\bibitem {Newman}
Newman, M.
\emph{Congruences for the coefficients of modular forms and some new congruences for the partition function}. 
Canad. J. Math. 9 (1957), 549–552.

\bibitem {Lazarev}
Lazarev, O.; Mizuhara, M. S.; Reid, B.; Swisher, H. 
\emph{Extension of a proof of the Ramanujan congruences for multipartitions}. 
Ramanujan J. 45 (2018), no. 1, 1–20.

\bibitem {Algorithm2}
Radu, S.
\emph{An algorithmic approach to Ramanujan's congruences}. 
Ramanujan J. 20 (2009), no. 2, 215–251. 

\bibitem {Gandhi}
Gandhi, J. M.
\emph{Simple proofs for the Ramanujan congruences p(5m+4)≡0 (mod 5) and p(7m+5)≡0 (mod 7)}. Recent Progress in Combinatorics (1969), pp. 221–222 Academic Press, New York 

\bibitem {Gandhi2}
Gandhi, J. M.
\emph{Generalization of Ramanujan's congruences p(5m+4)≡0(mod5) and p(7m+5)≡0(mod7)}. 
Monatsh. Math. 69 (1965), 389–392. 

\bibitem {Gandhi3}
Gandhi, J. M.
\emph{Congruences for pr(n) and Ramanujan's τ function}. 
Amer. Math. Monthly 70 (1963), 265–274.

\bibitem {Gandhi4}
Gandhi, J. M.
\emph{Congruences for σ(n)}. 
Proc. Nat. Acad. Sci. India Sect. A 37 (1967), 193–194. 

\bibitem {Ramanathan}
Ramanathan, K. G.
\emph{Congruence properties of σ(n), the sum of the divisors of n}. 
Math. Student 11, (1943), 33–35.

\bibitem {Gupta}
Gupta, H.
\emph{Congruence properties of σ(n)}. 
Math. Student 13, (1945), 25–29. 

\bibitem {Bonciocat}
Bonciocat, N. C.
\emph{Congruences for the convolution of divisor sum function}. 
Bull. Greek Math. Soc. 47 (2003), 19–29. 

\bibitem {RefBonc}
Gallardo, L. H.
\emph{On Bonciocat's congruences involving the sum of divisors function}.
Bull. Greek Math. Soc. 53 (2007), 69–70. 

\bibitem {Hahn}
Hahn, H.
\emph{Convolution sums of some functions on divisors}. 
Rocky Mountain J. Math. 37 (2007), no. 5, 1593–1622. 

\bibitem {17powers}
Hughes, K.
\emph{Ramanujan congruences for p−k(n) modulo powers of 17}. 
Canad. J. Math. 43 (1991), no. 3, 506–525. 

\bibitem{Chen}
Chen, W. Y. C.; Du, D. K.; Hou, Q.; Sun, L. H.
\emph{Congruences of multipartition functions modulo powers of primes}.
Ramanujan J. 35 (2014), no. 1, 1–19. 

\bibitem{Dazhao}
Tang, D.
\emph{Congruences modulo powers of 5 for k-colored partitions}. 
J. Number Theory 187 (2018), 198–214. 

\bibitem{Wang}
Wang, L.
\emph{Congruences modulo powers of 11 for some partition functions}. 
Proc. Amer. Math. Soc. 146 (2018), no. 4, 1515–1528. 

\bibitem {Rosenfeld}
Gustavson, R.; Ovchinnikov, A.; Pogudin, G.
\emph{New order bounds in differential elimination algorithms}.  
J. Symbolic Comput. 85 (2018), 128–147. 

\bibitem {tdeg}
Zobnin, A. I.
\emph{Differential standard bases with respect to inverse lexicographic orderings}. 
J. Math. Sci. (N.Y.) 163 (2009), no. 5, 523–533. 

\bibitem{Andrews}
Andrews, G. E.
\emph{A survey of multipartitions: congruences and identities}. Surveys in number theory, 1–19, 
Dev. Math., 17, Springer, New York (2008).

\bibitem{Chu}
Chu, W.; Zhou, R. R.
\emph{Congruences of three multipartition functions}.
Math. Commun. 17 (2012), no. 2, 433–445.

\bibitem{Saikia}
Saikia, N.; Chetry, J.
\emph{Infinite families of congruences modulo 7 for Ramanujan's general partition function}. 
Ann. Math. Qué. 42 (2018), no. 1, 127–132. 

\bibitem{Uroz}
Urroz, J. J.
\emph{Congruences for the partition function in certain arithmetic progressions}.
Discrete Math. 211 (2000), no. 1-3, 275–280. 

\bibitem {Kim}
Kim, D.
\emph{Convolution sums of odd and even divisor functions}. 
Honam Math. J. 35 (2013), no. 3, 445–506. 

\bibitem {Kim2}
Kim, D.
\emph{Convolution sums of odd and even divisor functions II}.
Honam Math. J. 37 (2015), no. 2, 149–185. 

\bibitem {Kim3}
Kim, D.; Park, Y. K.
\emph{Bernoulli identities and combinatoric convolution sums with odd divisor functions}.  
Abstr. Appl. Anal. (2014)

\bibitem{Maple1}
See the MAPLE script "multiPartitions.txt", attached to this paper on arXiv as an ancillary file.

\bibitem{Maple2}
See the MAPLE script "sigmaConvolutions.txt", attached to this paper on arXiv as an ancillary file.

\bibitem{Maple3}
See the MAPLE script "computeBasis.txt", attached to this paper on arXiv as an ancillary file.

\bibitem{Maple4}
See the MAPLE script "groebnerBases.txt", attached to this paper on arXiv as an ancillary file.

\bibitem{C1}
See the C++ source code "generateLinearCombo.txt", attached to this paper on arXiv as an ancillary file.

\bibitem{C2}
See the C++ source code "generateWeightedConv.txt", attached to this paper on arXiv as an ancillary file.

\end{thebibliography}
\end{document}